\documentclass[a4paper,11pt, reqno]{amsart}
\usepackage{geometry}
\usepackage{fullpage}
\usepackage[parfill]{parskip}   
\usepackage{graphicx}
\usepackage{amssymb}
\usepackage{amsmath}
\usepackage{epstopdf}
\usepackage{amsmath}
\usepackage{cite}

\usepackage{amsmath,amsthm,amscd,amssymb}
\usepackage{latexsym}
\usepackage[colorlinks,citecolor=red,pagebackref,hypertexnames=false]{hyperref}
 
\numberwithin{equation}{section}

\theoremstyle{plain}
\newtheorem{theorem}{Theorem}[section]
\newtheorem{Def}[theorem]{Definition}

\newtheorem{proposition}[theorem]{Proposition}

\theoremstyle{definition}

\theoremstyle{remark}
\newtheorem{remark}[theorem]{Remark}

\newtheorem{case[theorem]}{Case}

\def \R{{\mathbb R}}
\def \N{{\mathbb N}}

\def \C{{\mathbb C}}

\def\H{{\mathbb H}}
\def\supp{\hbox{supp\,}}
\def\norm#1.#2.{\lVert#1\rVert_{#2}}

\def\R{\mathbb R}

\def \M{{\mathcal M}}
\def \S{{\mathcal S}}
\def \W{{\mathcal W}}

\def \H{{\mathcal H}}

\title[Localization operators associated with the windowed Opdam--Cherednik transform]{Localization operators associated with the windowed Opdam--Cherednik transform on modulation spaces}

\author{Anirudha Poria}
\thanks{Research supported by ERC Starting Grant No. 713927.}

\address{Department of Mathematics, Bar-Ilan University, Ramat-Gan 5290002, Israel}

\email{anirudhamath@gmail.com}

\keywords{Opdam--Cherednik transform;  windowed Opdam--Cherednik transform; modulation spaces; localization operators; Schatten--von Neumann class; compact operators.}

\subjclass[2010]{Primary 47G30; Secondary 44A15, 42B35, 47B10.}

\date{\today}
\begin{document}
\maketitle

\begin{abstract} 
In this paper, we study a class of pseudodifferential operators known as time-frequency localization operators, which depend on a symbol $\varsigma$ and two windows functions $g_1$ and $g_2$. We first present some basic properties of the windowed Opdam--Cherednik transform. Then, we use modulation spaces associated with the Opdam--Cherednik transform as appropriate classes for symbols and windows, and study the boundedness and compactness of the localization operators associated with the windowed Opdam--Cherednik transform on modulation spaces. Finally, we show that these operators are in the Schatten--von Neumann class. 
\end{abstract}

\section{Introduction}  
Time-frequency localization operators are a mathematical tool to define a restriction of functions to a region in the time-frequency plane to extract time-frequency features. The role of these operators to localize a signal simultaneously in time and frequency domains, this can be seen as the uncertainty principle. The localization operators were introduced and studied by Daubechies \cite{dau88, dau90, pau88}, Ramanathan and Topiwala \cite{ram93}, and extensively investigated in \cite{fei02, won99, won02}. This class of operators occurs in various branches of pure and applied mathematics and has been studied by many authors. Localization operators are recognized as an important new mathematical tool and have found many applications to time-frequency analysis, quantum mechanics, the theory of differential equations, and signal processing (see \cite{cor03, mar02, now02, gro01, ram93, won02}). They are also known as anti-Wick operators, wave packets, Toeplitz operators, or Gabor multipliers (see \cite{now02, ber71, cor78, fei02}). For a detailed study of the theory of localization operators, we refer to the series of papers of Wong \cite{bog04, won2001, won2002, won2003}, and also the book of Wong \cite{won02}. Considerable attention has been devoted to study the localization operators to new contexts. For instance, localization operators were investigated in \cite{bac16} for the wavelet transform associated to the Riemann--Liouville operator, in \cite{bac2016} for the windowed Hankel transform and in \cite{mej20} for the k-Hankel wavelet transform. However, upto our knowledge, the localization operators have not been studied for the Opdam--Cherednik transform on modulation spaces. In this paper, we attempt to study the localization operators for the windowed Opdam--Cherednik transform on modulation spaces associated with the Opdam--Cherednik transform. 

The motivation to study the localization operators for the windowed Opdam--Cherednik transform on modulation spaces arises from the classical results on these operators for the short-time Fourier transform. Since the last decade modulation spaces have found to be very fruitful in various current trends (e.g., pseudo-differential operators,  partial differential equations, etc..) of investigation and have been widely used in several fields in analysis, physics and engineering.  Localization operators have implications in two main areas: quantum mechanics and signal analysis, and modulation spaces are widely used in these areas. We hope that the study of these operators on the modulation spaces makes a significant impact in these areas. Another important motivation to study the Jacobi--Cherednik operators arises from their relevance in the algebraic description of exactly solvable quantum many-body systems of Calogero--Moser--Sutherland type (see \cite{die00, hik96}) and they provide a useful tool in the study of special functions with root systems (see \cite{dun92, hec91}). These describe algebraically integrable systems in one dimension and have gained considerable interest in mathematical physics. Other motivation for the investigation of the Jacobi--Cherednik operator and the Opdam--Cherednik transform is to generalize the previous subjects which are bound with physics. For a more detailed discussion, we refer to \cite{mej16}.

Time-frequency localization operators were defined using the Schr\"odinger representation and the short-time Fourier transform, which suggests studying these operators as a part of time-frequency analysis. Modulation spaces were used as the appropriate function spaces for understanding these operators, as these spaces are associated to the short-time Fourier transform. However, without using the short-time Fourier transform, we define the localization operator using the windowed Opdam--Cherednik transform. As the harmonic analysis associated with the Opdam--Cherednik transform has known remarkable development (see \cite{and15, mej14, opd95, opd00, sch08}), the natural question to ask whether there exists the equivalent of the theory of localization operators in the framework of the Opdam--Cherednik transform. In this paper, we mainly concern the windowed Opdam--Cherednik transform under the setting of the Opdam--Cherednik transform. To measure the time-frequency concentration of functions and distributions, we use norms and function spaces that are associated with the Opdam--Cherednik transform. Here, we consider the modulation spaces associated with the Opdam--Cherednik transform, as the standard modulation spaces are not suited to this transform. Our main aim in this paper is to expose and study the boundedness and compactness of the localization operators associated with the windowed Opdam--Cherednik transform under suitable conditions on symbols and windows and show that these operators are in the Schatten--von Neumann class.

The paper is organized as follows. In Section \ref{sec2}, we recall some basic facts about the Jacobi--Cherednik operator and we discuss the main results for the Opdam--Cherednik transform. In Section \ref{sec3}, we discuss the modulation spaces associated with the Opdam--Cherednik transform. In Section \ref{sec4}, we define and study the windowed Opdam--Cherednik transform. Also, we present some basic properties of the windowed Opdam--Cherednik transform and show that this transform shares many properties with the short-time Fourier transform. In particular, we prove Plancherel's formula, orthogonality relation, and provide a reconstruction formula. Finally, in Section \ref{sec5}, we use modulation spaces associated with the Opdam--Cherednik transform as appropriate classes for symbols and windows, and study the boundedness and compactness of the localization operators associated with the windowed Opdam--Cherednik transform on modulation spaces. Also, we show that these operators are in the Schatten--von Neumann class.

\section{Harmonic analysis and the Opdam--Cherednik transform}\label{sec2}

In this section, we collect the necessary definitions and results from the harmonic analysis related to the Opdam--Cherednik transform. The main references for this section are \cite{and15, mej14, opd95, opd00, sch08}. However, we will use the same notation as in \cite{joh15, por21}.

Let $T_{\alpha, \beta}$ denote the Jacobi--Cherednik differential--difference operator (also called the Dunkl--Cherednik operator)
\[T_{\alpha, \beta} f(x)=\frac{d}{dx} f(x)+ \Big[ 
(2\alpha + 1) \coth x + (2\beta + 1) \tanh x \Big] \frac{f(x)-f(-x)}{2} - \rho f(-x), \]
where $\alpha, \beta$ are two parameters satisfying $\alpha \geq \beta \geq -\frac{1}{2}$ and $\alpha > -\frac{1}{2}$, and $\rho= \alpha + \beta + 1$. Let $\lambda \in \C$. The Opdam hypergeometric functions $G^{\alpha, \beta}_\lambda$ on $\R$ are eigenfunctions $T_{\alpha, \beta} G^{\alpha, \beta}_\lambda(x)=i \lambda  G^{\alpha, \beta}_\lambda(x)$ of $T_{\alpha, \beta}$ that are normalized such that $G^{\alpha, \beta}_\lambda(0)=1$. The eigenfunction $G^{\alpha, \beta}_\lambda$ is given by
\[G^{\alpha, \beta}_\lambda (x)= \varphi^{\alpha, \beta}_\lambda (x) - \frac{1}{\rho - i \lambda} \frac{d}{dx}\varphi^{\alpha, \beta}_\lambda (x)=\varphi^{\alpha, \beta}_\lambda (x)+ \frac{\rho+i \lambda}{4(\alpha+1)} \sinh 2x \; \varphi^{\alpha+1, \beta+1}_\lambda (x),  \]
where $\varphi^{\alpha, \beta}_\lambda (x)={}_2F_1 \left(\frac{\rho+i \lambda}{2}, \frac{\rho-i \lambda}{2} ; \alpha+1; -\sinh^2 x \right) $ is the hypergeometric function.

For every $ \lambda \in \C$ and $x \in  \R$, the eigenfunction
$G^{\alpha, \beta}_\lambda$ satisfy
\[ |G^{\alpha, \beta}_\lambda(x)| \leq C \; e^{-\rho |x|} e^{|\text{Im} (\lambda)| |x|},\] 
where $C$ is a positive constant. Since $\rho > 0$, we have
\begin{equation}\label{eq1}
|G^{\alpha, \beta}_\lambda(x)| \leq C \; e^{|\text{Im} (\lambda)| |x|}. 
\end{equation}
Let us denote by $C_c (\R)$ the space of continuous functions on $\R$ with compact support. The Opdam--Cherednik transform is the Fourier transform in the trigonometric Dunkl setting, and it is defined as follows.
\begin{Def}
Let $\alpha \geq \beta \geq -\frac{1}{2}$ with $\alpha > -\frac{1}{2}$. The Opdam--Cherednik transform $\mathcal{H}_{\alpha, \beta} (f)$ of a function $f \in C_c(\R)$ is defined by
\[ \H_{\alpha, \beta} (f) (\lambda)=\int_{\R} f(x)\; G^{\alpha, \beta}_\lambda(-x)\; A_{\alpha, \beta} (x) dx \quad \text{for all } \lambda \in \C, \] 
where $A_{\alpha, \beta} (x)= (\sinh |x| )^{2 \alpha+1} (\cosh |x| )^{2 \beta+1}$. The inverse Opdam--Cherednik transform for a suitable function $g$ on $\R$ is given by
\[ \H_{\alpha, \beta}^{-1} (g) (x)= \int_{\R} g(\lambda)\; G^{\alpha, \beta}_\lambda(x)\; d\sigma_{\alpha, \beta}(\lambda) \quad \text{for all } x \in \R, \]
where $$d\sigma_{\alpha, \beta}(\lambda)= \left(1- \dfrac{\rho}{i \lambda} \right) \dfrac{d \lambda}{8 \pi |C_{\alpha, \beta}(\lambda)|^2}$$ and 
$$C_{\alpha, \beta}(\lambda)= \dfrac{2^{\rho - i \lambda} \Gamma(\alpha+1) \Gamma(i \lambda)}{\Gamma \left(\frac{\rho + i \lambda}{2}\right)\; \Gamma\left(\frac{\alpha - \beta+1+i \lambda}{2}\right)}, \quad \lambda \in \C \setminus i \mathbb{N}.$$
\end{Def}

The Plancherel formula is given by 
\begin{equation}\label{eq03}
\int_{\R} |f(x)|^2 A_{\alpha, \beta}(x) dx=\int_\R \H_{\alpha, \beta} (f)(\lambda) \overline{\H_{\alpha, \beta} ( \check{f})(-\lambda)} \; d \sigma_{\alpha, \beta} (\lambda),
\end{equation}
where $\check{f}(x):=f(-x)$. Since the hypergeometric functions $G^{\alpha, \beta}_\lambda$  satisfy
$G^{\alpha, \beta}_{\lambda }(tx)=G^{\alpha, \beta}_{\lambda  t}(x)$, for every $ \lambda \in \C $ and $x, t \in  \R$, using the definition of $\mathcal{H}_{\alpha, \beta}$, we obtain $\mathcal{H}_{\alpha, \beta}(\check{f})(-\lambda)=\mathcal{H}_{\alpha, \beta}(f)(\lambda)$. Therefore, we can rewrite the Plancherel formula (\ref{eq03}) as follows: 
\begin{equation}\label{newpf}
\int_{\mathbb{R}}|f(x)|^{2} A_{\alpha, \beta}(x) d x=\int_{\mathbb{R}} \left| \mathcal{H}_{\alpha, \beta}(f)(\lambda) \right|^2 d \sigma_{\alpha, \beta}(\lambda).
\end{equation}

Let $L^p(\R,A_{\alpha, \beta} )$ (resp. $L^p(\R, \sigma_{\alpha, \beta} )$), $p \in [1, \infty] $, denote the $L^p$-spaces corresponding to the measure $A_{\alpha, \beta}(x) dx$ (resp. $d | \sigma_{\alpha, \beta} |(x)$). The Schwartz space $\S_{\alpha, \beta}(\R )=(\cosh x )^{-\rho} \S(\R)$ is defined as the space of all differentiable functions $f$ such that 
$$ \sup_{x \in \R} \; (1+|x|)^m e^{\rho |x|} \left|\frac{d^n}{dx^n} f(x) \right|<\infty,$$ 
for all $m, n \in \N_0 = \N \cup \{0\}$, equipped with the obvious seminorms. The Opdam--Cherednik transform $\H_{\alpha, \beta}$ and its inverse $\H_{\alpha, \beta}^{-1}$ are topological isomorphisms between the space $\S_{\alpha, \beta}(\R )$ and the space $\S(\R)$ (see \cite{sch08}, Theorem 4.1).

The generalized translation operator associated with the Opdam--Cherednik transform is defined by \cite{ank12}
\begin{equation}\label{eq04}
\tau_x^{(\alpha, \beta)} f(y)=\int_{\R} f(z) \; {d\mu}_{\;x, y}^{(\alpha, \beta)}(z),
\end{equation}
where ${d\mu}_{\;x, y}^{(\alpha, \beta)}$ is given by 
\begin{equation}\label{eq05}
{d\mu}_{\;x, y}^{(\alpha, \beta)}(z)=
\begin{cases} 
\mathcal{K}_{\alpha, \beta}(x,y,z)\; A_{\alpha, \beta}(z)\; dz & \text{if} \;\; xy \neq 0 \\
d \delta_x (z) & \text{if} \;\; y=0 \\
d \delta_y (z) & \text{if} \;\; x=0
\end{cases}  
\end{equation}
and
\begin{equation*}
\begin{aligned}
\mathcal{K}_{\alpha, \beta} {} & (x,y,z) 
=M_{\alpha, \beta} |\sinh x \cdot \sinh y \cdot \sinh z  |^{-2 \alpha} \int_0^\pi g(x, y, z, \chi)_+^{\alpha-\beta-1} 
\\
& \times \left[ 1 - \sigma^\chi_{x,y,z}+ \sigma^\chi_{x,z,y} + \sigma^\chi_{z,y,x} + \frac{\rho}{\beta+\frac{1}{2}} \coth x \cdot \coth y \cdot \coth z (\sin \chi)^2 \right] \times (\sin \chi)^{2 \beta}\; d \chi,
\end{aligned}
\end{equation*}
where \[ M_{\alpha, \beta}=\frac{\Gamma(\alpha+1)}{\sqrt{\pi} \Gamma(\alpha-\beta) \Gamma\left(\beta+\frac{1}{2}\right)} ,\] 
if $x, y, z \in \R \setminus \{0\} $ satisfy the triangular inequality  $||x|-|y||<|z|<|x|+|y|$, and $\mathcal{K}_{\alpha, \beta} (x,y,z) =0$ otherwise. Here 
\[ \sigma^\chi_{x,y,z}=
\begin{cases} 
\frac{\cosh x \cdot \cosh y - \cosh z \cdot \cos \chi}{\sinh x \cdot \sinh y} & \text{if} \;\; xy \neq 0 \\
0 & \text{if} \;\; xy = 0 
\end{cases}
\quad \text{for} \; x, y, z \in \R, \; \chi \in [0,\pi], \]
$g(x, y, z, \chi) = 1- \cosh^2 x - \cosh^2 y - \cosh^2 z + 2 \cosh x \cdot \cosh y \cdot \cosh z \cdot \cos \chi$, and 
\[ g_+=
\begin{cases} 
g &  \text{if} \;\;  g> 0 \\
0 & \text{if} \;\;  g \leq 0.
\end{cases} \]
The kernel $\mathcal{K}_{\alpha, \beta} (x,y,z)$ satisfies the following symmetry properties:
   
$\mathcal{K}_{\alpha, \beta} (x,y,z)=\mathcal{K}_{\alpha, \beta} (y,x,z)$, $\mathcal{K}_{\alpha, \beta} (x,y,z)=\mathcal{K}_{\alpha, \beta} (-z,y,-x)$, $\mathcal{K}_{\alpha, \beta} (x,y,z)=\mathcal{K}_{\alpha, \beta} (x,-z,-y)$.

Next, we recall some basic properties of $\tau_{x}^{(\alpha, \beta)}$ from \cite{ank12}. For every $x,y \in \R$, we have
\begin{equation}\label{eq06}
\tau_x^{(\alpha, \beta)} f(y) = \tau_y^{(\alpha, \beta)} f(x),
\end{equation}
and 
\begin{equation}\label{eq07}
\H_{\alpha, \beta}(\tau_x^{(\alpha, \beta)} f )(\lambda)=G_\lambda^{\alpha, \beta} (x) \; \H_{\alpha, \beta} (f)(\lambda),
\end{equation}
for $f \in C_c(\R)$.

If $f \in L^1(\R,A_{\alpha, \beta} )$, then
\begin{eqnarray}\label{eq08}
\int_\R \tau_x^{(\alpha, \beta)} f(y) \; A_{\alpha, \beta}(y) \; dy 
& = & \int_\R \left( \int_{\R} f(z) \; {d\mu}_{\;x, y}^{(\alpha, \beta)}(z) \right) \; A_{\alpha, \beta}(y) \; dy \nonumber \\
&=& \int_\R \left( \int_{\R} f(z) \; \mathcal{K}_{\alpha, \beta}(x,y,z)\; A_{\alpha, \beta}(z)\; dz \right) \; A_{\alpha, \beta}(y) \; dy \nonumber \\
&=& \int_\R f(z) \left( \int_{\R} \; \mathcal{K}_{\alpha, \beta} (x,-z,-y) \; A_{\alpha, \beta}(y) \; dy \right) \; A_{\alpha, \beta}(z)\; dz \nonumber \\
&=& \int_\R f(z) \left( \int_{\R} \; \mathcal{K}_{\alpha, \beta} (x,-z,y) \; A_{\alpha, \beta}(y) \; dy \right) \; A_{\alpha, \beta}(z)\; dz \nonumber \\
&=& \int_\R f(z) \left( \int_{\R} \; 
{d\mu}_{\;x, -z}^{(\alpha, \beta)}(y) \right) \; A_{\alpha, \beta}(z)\; dz \nonumber \\  
&=& \int_\R f(y) \; A_{\alpha, \beta}(y)\; dy \qquad \qquad \left(\text{since} \int_{\R} \;  {d\mu}_{\;x, -z}^{(\alpha, \beta)}(y)=1 \right). \qquad \qquad
\end{eqnarray}
For every $f \in L^p(\R,A_{\alpha, \beta} )$ and every $x \in \R$, the function $\tau_x^{(\alpha, \beta)} f$ belongs to the space $L^p(\R,A_{\alpha, \beta} )$ and 
\begin{equation}\label{eq09}
\left\Vert \tau_x^{(\alpha, \beta)} f \right\Vert_{L^p(\R,A_{\alpha, \beta} )} \leq C_{\alpha, \beta} \left\Vert f \right\Vert_{L^p(\R,A_{\alpha, \beta} )}, 
\end{equation}
where $C_{\alpha, \beta}$ is a positive constant.

The convolution product associated with the Opdam--Cherednik transform is defined for two suitable functions $f$ and $g$  by \cite{ank12}
\[ (f *_{\alpha, \beta} g) (x)=\int_\R \tau_x^{(\alpha, \beta)} f(-y) \;g(y) \; A_{\alpha, \beta}(y) \; dy  \]
and
\begin{equation}\label{eq10}
\H_{\alpha, \beta} (f *_{\alpha, \beta} g)= \H_{\alpha, \beta} (f) \; \H_{\alpha, \beta} (g).
\end{equation}

\section{Modulation spaces associated with the Opdam--Cherednik transform}\label{sec3}
The modulation spaces were introduced by Feichtinger \cite{fei03, fei97}, by imposing integrability conditions on the short-time Fourier transform (STFT) of tempered distributions. More specifically, for $x, w \in \R$, let $M_w$ and $T_x$ denote the operators of modulation and translation. Then, the STFT of a function $f$ with respect to a window function $g \in  \S(\R)$ is defined by
\[ V_g f (x,w)=\langle f, M_w T_x g \rangle=\int_{\R} f(t) \overline{g(t-x)} e^{-2 \pi i w t} dt, \quad (x,y) \in \R^2. \] 

Here we are interested in modulation spaces with respect to measure $A_{\alpha, \beta}(x) dx$.  
\begin{Def}
Fix a non-zero window $g \in \mathcal{S}(\R)$, and $1 \leq p,q \leq \infty$. Then the modulation space $M^{p,q}(\R, A_{\alpha, \beta})$  consists of all tempered distributions $f \in \mathcal{S'}(\R)$ such that $V_g f \in L^{p,q}(\R^2, A_{\alpha, \beta})$. The norm on $M^{p,q}(\R, A_{\alpha, \beta})$ is 
\begin{eqnarray*}
\Vert f \Vert_{M^{p,q}(\R, A_{\alpha, \beta})}
&=& \Vert V_g f \Vert_{L^{p,q}(\R^2, A_{\alpha, \beta})} \\
&=& \bigg( \int_{\R} \bigg( \int_{\R} |V_gf(x,w)|^p A_{\alpha, \beta}(x) dx \bigg)^{q/p} A_{\alpha, \beta}(w) dw \bigg)^{1/q} < \infty,
\end{eqnarray*}
with the usual adjustments if $p$ or $q$ is infinite. If $p=q$, then we write $M^p(\R, A_{\alpha, \beta})$ instead of $M^{p,p}(\R, A_{\alpha, \beta})$. Also,  we denote by $M^p(\R, \sigma_{\alpha, \beta})$ the modulation space corresponding to the measure $d |\sigma_{\alpha, \beta}|(x)$ and $M^p(\R)$ the modulation space corresponding to the Lebesgue measure $dx$.   
\end{Def}
The definition of $M^{p, q}(\R, A_{\alpha, \beta})$ is independent of the choice of $g$ in the sense that each different choice of $g$ defines an equivalent norm on $M^{p, q}(\R, A_{\alpha, \beta})$. Each modulation space is a Banach space. For $p=q=2$, we have that $M^2(\R, A_{\alpha, \beta}) =L^2(\R, A_{\alpha, \beta}).$ For other $p=q$, the space $M^p(\R, A_{\alpha, \beta})$ is not $L^p(\R, A_{\alpha, \beta})$. In fact for $p=q>2$, the space $M^p(\R, A_{\alpha, \beta})$ is a superset of $L^2(\R, A_{\alpha, \beta})$. Here, we collect some basic properties and inclusion relations of modulation spaces with respect to $A_{\alpha, \beta}$. We define the space of special windows $\mathcal{S}_{\mathcal{C}}(\mathbb{R})$ by 
\begin{eqnarray*}
&& \mathcal{S}_{\mathcal{C}} \left( \mathbb{R}  \right)=\left\{ f \in L^{2} (\mathbb{R} ) :  \; f=V_{g}^{*} F=\iint_{\mathbb{R}^{2 }} F(x, w)\; M_{w} T_{x} g  \; dx \; dw, \right. \\
&& \left. \qquad \qquad  \qquad \qquad  \qquad  \text{where} \; F \in L^{\infty} (\mathbb{R}^{2}) \; \text{and} \; \supp F  \; \text{is compact} \right\}.
\end{eqnarray*}
Then $\mathcal{S}_{\mathcal{C}}\left(\mathbb{R}\right) \subseteq \mathcal{S}\left(\mathbb{R}\right)$ and 
$\mathcal{S}_{\mathcal{C}}\left(\mathbb{R}\right)$ is dense in $M^1(\R, A_{\alpha, \beta})$ (see \cite{gro01}, Lemma 11.4.1). Let $B$ be a Banach space of tempered distributions with the following properties: (1) $B$ is invariant under time-frequency shifts, and $\left\|T_{x} M_{w} f\right\|_{B} \leq C  \|f\|_{B}$ for all $f \in B$, (2) $M^1(\R, A_{\alpha, \beta}) \cap B \neq\{0\}$. Then $M^1(\R, A_{\alpha, \beta})$ is embedded in $B$ (see \cite{gro01}, Theorem 12.1.9). Also, $M^p(\R, A_{\alpha, \beta})$ is invariant under time-frequency shifts and $\left\|T_{x} M_{w} f\right\|_{M^p(\R, A_{\alpha, \beta})} \leq C  \|f\|_{M^p(\R, A_{\alpha, \beta})}$. 
Since $\mathcal{S}_{\mathcal{C}}\left(\mathbb{R} \right) \subseteq M^1(\R, A_{\alpha, \beta}) \cap M^p(\R, A_{\alpha, \beta})$, using Corollary 12.1.10 from \cite{gro01}, we obtain  the following inclusions
\[ \mathcal{S}(\R) \subset M^1(\R, A_{\alpha, \beta}) \subset M^2(\R, A_{\alpha, \beta})=L^2(\R, A_{\alpha, \beta}) \subset M^\infty(\R, A_{\alpha, \beta}) \subset \mathcal{S'}(\R). \]
In particular, we have $M^p(\R, A_{\alpha, \beta}) \hookrightarrow L^p(\R, A_{\alpha, \beta})$ for $1 \leq p \leq 2$, and  $L^p(\R, A_{\alpha, \beta}) \hookrightarrow M^p(\R, A_{\alpha, \beta})$ for $2 \leq p \leq \infty$. Furthermore, the dual of a modulation space is also a modulation space, if $p < \infty$, $q < \infty$, $(M^{p, q}(\R, A_{\alpha, \beta}))^{'} =M^{p', q'}(\R, A_{\alpha, \beta})$, where $p', \; q'$ denote the dual exponents of $p$ and $q$, respectively. We refer to Gr\"ochenig's book \cite{gro01} for further properties and uses of modulation spaces. 

\section{The windowed Opdam--Cherednik transform}\label{sec4}

Let $g \in L^2(\R,A_{\alpha, \beta} )$ and $\xi \in \R$, the modulation operator of $g$ associated with the Opdam--Cherednik transform is defined by 
\begin{equation}\label{eq11}
\M^{(\alpha, \beta)}_\xi g=\H^{-1}_{\alpha, \beta} \left( \sqrt{\tau_\xi^{(\alpha, \beta)} |\H_{\alpha, \beta}(g)|^2 } \right) .
\end{equation}
Then, for every $g \in L^2(\R,A_{\alpha, \beta} )$ and $\xi \in \R$, by using the Plancherel formula (\ref{eq03}) and 
the translation invariance of the Plancherel measure 
$d \sigma_{\alpha, \beta}$,  we obtain
\begin{equation}\label{eq12}
\left\Vert \M^{(\alpha, \beta)}_\xi g \right\Vert_{L^2(\R,A_{\alpha, \beta} )}=\left\Vert g \right\Vert_{L^2(\R,A_{\alpha, \beta} )}.
\end{equation}
Now, for a non-zero window function $g \in L^2(\R,A_{\alpha, \beta} )$ and $(x, \xi) \in \R^2$, we consider the function $g_{x, \xi}^{(\alpha, \beta)}$ defined by 
\begin{equation}\label{eq13}
g_{x, \xi}^{(\alpha, \beta)}= \tau_x^{(\alpha, \beta)} \M^{(\alpha, \beta)}_\xi g.
\end{equation}
For any function $f \in L^2(\R,A_{\alpha, \beta} )$, we define the windowed Opdam--Cherednik transform by
\begin{equation}\label{eq14}
\W^{(\alpha, \beta)}_g(f)(x,\xi)=\int_\R f(s) \; \overline{g_{x, \xi}^{(\alpha, \beta)}(-s)} \;  A_{\alpha, \beta}(s) \; ds, \quad (x,\xi) \in \R^2,
\end{equation}
which can be also written in the form
\begin{equation}\label{eq15}
\W^{(\alpha, \beta)}_g(f)(x,\xi)= \left(f *_{\alpha,\beta} \overline{ \M^{(\alpha, \beta)}_\xi g} \right)(x).
\end{equation}
We define the measure $A_{\alpha, \beta} \otimes \sigma_{\alpha, \beta}$ on $\R^2$ by 
\begin{equation}\label{eq16}
d(A_{\alpha, \beta} \otimes \sigma_{\alpha, \beta})(x, \xi)= A_{\alpha, \beta}(x)  dx \; d\sigma_{\alpha, \beta}(\xi).
\end{equation}

The windowed Opdam--Cherednik transform satisfies the following properties. 
\begin{proposition}
Let $g \in L^2(\R,A_{\alpha, \beta} )$ be a non-zero window function.  Then we have

$(1)$ $($Plancherel's formula$)$ For every $f \in L^2(\R,A_{\alpha, \beta})$, 
\begin{equation}\label{eq17}
\left\Vert \W^{(\alpha, \beta)}_g(f) \right\Vert_{L^2(\R^2, \; A_{\alpha, \beta}\otimes \sigma_{\alpha, \beta})} = \Vert f \Vert_{L^2(\R,A_{\alpha, \beta})} \; \Vert g \Vert_{L^2(\R,A_{\alpha, \beta})}.
\end{equation}

$(2)$ $($Orthogonality relation$)$  For every $f , h \in L^2(\R,A_{\alpha, \beta})$, we have
\begin{equation}\label{eq18}
\iint_{\R^2} \W^{(\alpha, \beta)}_g(f)(x, \xi) \; \overline{ \W^{(\alpha, \beta)}_g(h)(x, \xi)} \; d(A_{\alpha, \beta} \otimes \sigma_{\alpha, \beta})(x, \xi)  = \Vert g \Vert^2_{L^2(\R,A_{\alpha, \beta})} \int_\R f(s) \overline{h(s)} \;A_{\alpha, \beta}(s) \;ds.
\end{equation}

$(3)$ $($Reproducing kernel Hilbert space$)$ The space $\W^{(\alpha, \beta)}_g(L^2(\R,A_{\alpha, \beta}))$ is a reproducing kernel Hilbert space in $L^2(\R^2,A_{\alpha, \beta} \otimes \sigma_{\alpha, \beta} )$ with kernel function $K_g$ defined by 
\begin{eqnarray}\label{eq19}
K_g ((x', \xi');(x, \xi))
& = & \frac{1}{\Vert g \Vert^2_{L^2(\R,A_{\alpha, \beta})}} \; \left( g_{x, \xi}^{(\alpha, \beta)} (- \; \cdot) *_{\alpha, \beta} \overline{  \M^{(\alpha, \beta)}_{\xi'} g} \right)(x') \nonumber \\ 
& = & \frac{1}{\Vert g \Vert^2_{L^2(\R,A_{\alpha, \beta})}} \; \W^{(\alpha, \beta)}_g \left( g_{x, \xi}^{(\alpha, \beta)} (- \; \cdot) \right) (x', \xi').
\end{eqnarray}
Furthermore, the kernel is pointwise bounded
\begin{equation}\label{eq20}
\left| K_g ((x', \xi');(x, \xi)) \right| \leq C_{\alpha, \beta}, \quad \text{for all}\; (x, \xi);   (x', \xi') \in \R^2,
\end{equation}
where $C_{\alpha, \beta}$ is a positive constant. 
\end{proposition}
\begin{proof}
(1) Using Plancherel's formula (\ref{newpf}), and relations (\ref{eq10}) and (\ref{eq15}), we obtain 
\begin{eqnarray*}
&& \iint_{\R^2} \left|\W^{(\alpha, \beta)}_g(f)(x, \xi) \right|^2 \; d(A_{\alpha, \beta} \otimes \sigma_{\alpha, \beta}) (x, \xi) \\
&&= \iint_{\R^2} \left| \left(f *_{\alpha,\beta} \overline{ \M^{(\alpha, \beta)}_\xi g} \right)(x) \right|^2 \; A_{\alpha, \beta}(x)  dx \; d\sigma_{\alpha, \beta}(\xi) \\
&&= \iint_{\R^2} \left| \H_{\alpha, \beta} \left(f *_{\alpha,\beta} \overline{ \M^{(\alpha, \beta)}_\xi g} \right)(\lambda) \right|^2 \;  d \sigma_{\alpha, \beta}(\lambda) \; d\sigma_{\alpha, \beta}(\xi) \\
&&= \iint_{\R^2} \left| \H_{\alpha, \beta} (f)(\lambda) \right|^2 \; \left| \H_{\alpha, \beta} \left( \overline{ \M^{(\alpha, \beta)}_\xi g} \right)(\lambda) \right|^2 \;  d \sigma_{\alpha, \beta}(\lambda) \; d\sigma_{\alpha, \beta}(\xi)\\
&&= \iint_{\R^2} \left| \H_{\alpha, \beta} (f)(\lambda) \right|^2 \; \tau^{(\alpha, \beta)}_\xi \left| \H_{\alpha, \beta} (g) \right|^2 (\lambda)\;  d \sigma_{\alpha, \beta}(\lambda) \; d\sigma_{\alpha, \beta}(\xi)\\
&&= \int_{\R} \left| \H_{\alpha, \beta} (f)(\lambda) \right|^2 \; \int_{\R} \tau^{(\alpha, \beta)}_\lambda \left| \H_{\alpha, \beta} (g) \right|^2 (\xi)\;  d \sigma_{\alpha, \beta}(\xi) \; d\sigma_{\alpha, \beta}(\lambda)\\
&&= \int_{\R} \left| \H_{\alpha, \beta} (f)(\lambda) \right|^2  d\sigma_{\alpha, \beta}(\lambda) \; \int_{\R} \left| \H_{\alpha, \beta} (g) (\xi) \right|^2   d \sigma_{\alpha, \beta}(\xi) \\
&&= \Vert f \Vert^2_{L^2(\R,A_{\alpha, \beta})} \; \Vert g \Vert^2_{L^2(\R,A_{\alpha, \beta})}.
\end{eqnarray*}

(2) Using the polarization identity and Plancherel's formula (\ref{eq17}), we obtain the result.

(3) From relations (\ref{eq14}) and (\ref{eq18}), we obtain that 
\begin{eqnarray*}
&& \W^{(\alpha, \beta)}_g(f)(x,\xi)\\
&& = \frac{1}{\Vert g \Vert^2_{L^2(\R,A_{\alpha, \beta})}} \iint_{\R^2} \W^{(\alpha, \beta)}_g(f)(x', \xi') \; \overline{ \W^{(\alpha, \beta)}_g \left(g_{x, \xi}^{(\alpha, \beta)}(- \; \cdot)\right)(x', \xi')}  \; d(A_{\alpha, \beta} \otimes \sigma_{\alpha, \beta})(x', \xi') \\
&& = \left\langle \W^{(\alpha, \beta)}_g(f), \;K_g((\cdot , \cdot);(x,\xi))  \right\rangle_{L^2(\R^2,A_{\alpha, \beta} \otimes \sigma_{\alpha, \beta} )} , 
\end{eqnarray*}
where $\langle \cdot , \cdot \rangle_{L^2(\R^2,A_{\alpha, \beta} \otimes \sigma_{\alpha, \beta} )}$ denotes the inner product of $L^2(\R^2,A_{\alpha, \beta} \otimes \sigma_{\alpha, \beta} )$ and $K_g$ is the function on $\R^2$ defined by  
\begin{eqnarray*}
K_g ((x', \xi');(x, \xi))
& = & \frac{1}{\Vert g \Vert^2_{L^2(\R,A_{\alpha, \beta})}} \; \W^{(\alpha, \beta)}_g \left( g_{x, \xi}^{(\alpha, \beta)} (- \; \cdot) \right) (x', \xi') \\
& = & \frac{1}{\Vert g \Vert^2_{L^2(\R,A_{\alpha, \beta})}} \; \left( g_{x, \xi}^{(\alpha, \beta)} (- \; \cdot) *_{\alpha, \beta} \overline{  \M^{(\alpha, \beta)}_{\xi'} g} \right)(x'). 
\end{eqnarray*}
Also, for every $(x, \xi) \in \R^2$, by using Plancherel's formula (\ref{eq17}), we obtain 
\[\left\Vert K_g((\cdot , \cdot);(x,\xi)) \right\Vert_{L^2(\R^2,A_{\alpha, \beta} \otimes \sigma_{\alpha, \beta} )} \leq C_{\alpha, \beta}.  \]
Finally, by the Cauchy--Schwarz inequality, for every $(x, \xi);  (x', \xi') \in \R^2$ we have
\begin{eqnarray*}
|K_g ((x', \xi');(x, \xi)) | 
& \leq & \frac{1}{\Vert g \Vert^2_{L^2(\R,A_{\alpha, \beta})}} \int_\R \left|g_{x, \xi}^{(\alpha, \beta)} (- s)\right| \; \left|g_{x', \xi'}^{(\alpha, \beta)} (- s) \right| \; A_{\alpha, \beta}(s) \; ds \\
& \leq &  \frac{1}{\Vert g \Vert^2_{L^2(\R,A_{\alpha, \beta})}} \; \left\Vert g_{x, \xi}^{(\alpha, \beta)} \right\Vert_{L^2(\R,A_{\alpha, \beta})} \left\Vert g_{x', \xi'}^{(\alpha, \beta)} \right\Vert_{L^2(\R,A_{\alpha, \beta})}         
= C_{\alpha, \beta}.
\end{eqnarray*}
This proves that the kernel $K_g \in  L^2(\R^2,A_{\alpha, \beta} \otimes \sigma_{\alpha, \beta})$ and is bounded.
\end{proof}
  
\begin{theorem}
Let $g \in L^2(\R,A_{\alpha, \beta} )$ be a non-zero window function. Then for every $f \in L^2(\R,A_{\alpha, \beta} )$, we have  

$(1)$   
\begin{equation}\label{eq21}
\left\Vert \W^{(\alpha, \beta)}_g(f) \right\Vert_{L^\infty(\R^2,A_{\alpha, \beta} \otimes \sigma_{\alpha, \beta})} \leq C_{\alpha, \beta} \Vert f \Vert_{L^2(\R,A_{\alpha, \beta})} \; \Vert g \Vert_{L^2(\R,A_{\alpha, \beta})}.
\end{equation}

$(2)$ The function $\W^{(\alpha, \beta)}_g(f) \in L^p(\R^2,A_{\alpha, \beta} \otimes \sigma_{\alpha, \beta})$, $p \in [2, \infty)$ and 
\begin{equation}\label{eq22}
\left\Vert \W^{(\alpha, \beta)}_g(f) \right\Vert_{L^p(\R^2,A_{\alpha, \beta} \otimes \sigma_{\alpha, \beta})} \leq C_{\alpha, \beta} \Vert f \Vert_{L^2(\R,A_{\alpha, \beta})} \; \Vert g \Vert_{L^2(\R,A_{\alpha, \beta})}.
\end{equation}  
\end{theorem}
\begin{proof}
(1) Using the Cauchy--Schwarz inequality, and relations (\ref{eq12}) and (\ref{eq13}), we get  
\[ \left|\W^{(\alpha, \beta)}_g(f)(x, \xi) \right| \leq \Vert f \Vert_{L^2(\R,A_{\alpha, \beta})} \; \left\Vert g_{x, \xi}^{(\alpha, \beta)} \right\Vert_{L^2(\R,A_{\alpha, \beta})} \leq C_{\alpha, \beta} \Vert f \Vert_{L^2(\R,A_{\alpha, \beta})} \; \Vert g \Vert_{L^2(\R,A_{\alpha, \beta})}.  \]
Therefore,
\[ \left\Vert \W^{(\alpha, \beta)}_g(f) \right\Vert_{L^\infty(\R^2,A_{\alpha, \beta} \otimes \sigma_{\alpha, \beta})} \leq C_{\alpha, \beta} \Vert f \Vert_{L^2(\R,A_{\alpha, \beta})} \; \Vert g \Vert_{L^2(\R,A_{\alpha, \beta})}.   \]

(2) Using Plancherel's formula (\ref{eq17}), relation (\ref{eq21}) and the Riesz--Thorin interpolation theorem (see \cite{ste56}), we obtain the result. 
\end{proof}

\begin{theorem}[Reconstruction formula] 
Let $g \in L^2(\R,A_{\alpha, \beta} )$ be a non-zero positive window function. Then for every $f \in L^2(\R,A_{\alpha, \beta} )$, we have
\[f(\cdot)= \frac{1}{\Vert g \Vert^2_{L^2(\R,A_{\alpha, \beta})}}  \iint_{\R^2} \W^{(\alpha, \beta)}_g(f)(x, \xi)\; g_{x, \xi}^{(\alpha, \beta)} (-\; \cdot) \; d(A_{\alpha, \beta} \otimes \sigma_{\alpha, \beta}) (x, \xi), \]
weakly in $L^2(\R,A_{\alpha, \beta} )$. 
\end{theorem}
\begin{proof}
Using relation (\ref{eq18}) and applying Fubini's theorem, for every $h \in L^2(\R,A_{\alpha, \beta} )$ we obtain 
\begin{eqnarray*}
&& \int_\R f(s) \overline{h(s)} \;A_{\alpha, \beta}(s) \;ds \\
& =& \frac{1}{\Vert g \Vert^2_{L^2(\R,A_{\alpha, \beta})}} \iint_{\R^2} \W^{(\alpha, \beta)}_g(f)(x, \xi) \; \overline{ \W^{(\alpha, \beta)}_g(h)(x, \xi)} \; d(A_{\alpha, \beta} \otimes \sigma_{\alpha, \beta})(x, \xi) \\
&=& \frac{1}{\Vert g \Vert^2_{L^2(\R,A_{\alpha, \beta})}} \int_\R \left( \iint_{\R^2} \W^{(\alpha, \beta)}_g(f)(x, \xi) \; g_{x, \xi}^{(\alpha, \beta)}(-s) \; d( A_{\alpha, \beta} \otimes \sigma_{\alpha, \beta})(x, \xi) \right) \; \overline{h(s)}  \;  A_{\alpha, \beta}(s) \; ds, 
\end{eqnarray*}
which completes the proof.
\end{proof}

\section{Localization operators for the windowed Opdam--Cherednik transform}\label{sec5}

In this section, we define the localization operators for the windowed Opdam--Cherednik transform and we show that these operators are bounded. Also, we prove that localization operators are compact and in the Schatten--von Neumann class. 

\begin{Def}
Let $\varsigma \in L^1(\R^2,A_{\alpha, \beta} \otimes \sigma_{\alpha, \beta})  \cup  L^\infty(\R^2,A_{\alpha, \beta} \otimes \sigma_{\alpha, \beta})$. The localization operator for the windowed Opdam--Cherednik transform associated with the symbol $\varsigma$ and two window functions $g_1$ and $g_2$, is denoted by  $\mathfrak{L}_{g_1, g_2}(\varsigma)$, and defined on $L^2(\R,A_{\alpha, \beta})$, by
\begin{equation}\label{eq41}
\mathfrak{L}_{g_1, g_2}(\varsigma) (f)(y)=\iint_{\R^2} \varsigma(x, \xi) \; \W^{(\alpha, \beta)}_{g_1}(f)(x, \xi)\; {g_2}_{x, \xi}^{(\alpha, \beta)} (- y) \; d(A_{\alpha, \beta} \otimes \sigma_{\alpha, \beta}) (x, \xi), \quad y \in \R. 
\end{equation}  
Also, it is useful to rewrite the definition of $\mathfrak{L}_{g_1, g_2}(\varsigma)$ in a weak sense as, for every $f,  h \in L^2(\R,A_{\alpha, \beta})$ 
\begin{equation}\label{eq42}
\left\langle \mathfrak{L}_{g_1, g_2}(\varsigma)(f) , h \right\rangle_{L^2(\R,A_{\alpha, \beta})} = \iint_{\R^2} \varsigma(x, \xi) \; \W^{(\alpha, \beta)}_{g_1}(f)(x, \xi)\; \overline{\W^{(\alpha, \beta)}_{g_2}(h)(x, \xi)} \; d(A_{\alpha, \beta} \otimes \sigma_{\alpha, \beta}) (x, \xi).
\end{equation}
\end{Def}

We denote by $\mathcal{B}(L^p(\R,A_{\alpha, \beta}))$, $1 \leq p \leq \infty$, the space of all bounded linear operators from $L^p(\R,A_{\alpha, \beta})$ into itself. In particular, $\mathcal{B}(L^2(\R,A_{\alpha, \beta}))$ denote the C$^*$-algebra of bounded linear operator $\mathcal{A}$ from $L^2(\R,A_{\alpha, \beta})$ into itself, equipped with the norm 
\[ \Vert \mathcal{A} \Vert_{\mathcal{B}(L^2(\R,A_{\alpha, \beta}))}= \sup_{\Vert f \Vert_{L^2(\R,A_{\alpha, \beta})} \leq 1} \Vert \mathcal{A}(f) \Vert_{L^2(\R,A_{\alpha, \beta})}. \]

Next, we define the Schatten--von Neumann class $S_p$. For a compact operator $\mathcal{A} \in \mathcal{B}(L^2(\R,A_{\alpha, \beta}))$, the eigenvalues of the positive self-adjoint operator $|\mathcal{A}|=\sqrt{\mathcal{A}^* \mathcal{A}}$ are called the singular values of $\mathcal{A}$ and denoted by $\{ s_n(\mathcal{A}) \}_{n \in \mathbb{N}}$. For $1 \leq p < \infty$, the Schatten--von Neumann class $S_p$ is defined to be the space of all compact operators whose singular values lie in $\ell^p$. $S_p$ is equipped with the norm 
\[ \Vert \mathcal{A} \Vert_{S_p}= \left(\sum_{n=1}^\infty  (s_n(\mathcal{A}))^p \right)^{1/p}. \]
For $p=\infty$,  the Schatten--von Neumann class $S_\infty$ is the class of all compact operators with the norm  $\Vert \mathcal{A} \Vert_{S_\infty} :=\Vert \mathcal{A} \Vert_{\mathcal{B}(L^2(\R,A_{\alpha, \beta}))}$. In particular, for $p = 1$, we define the trace of an operator $\mathcal{A}$ in $S_1$ by 
\[ tr(\mathcal{A}) = \sum_{n=1}^\infty \langle \mathcal{A}  v_n, v_n  \rangle_{L^2(\R,A_{\alpha, \beta})}, \]
where $\{v_n \}_n$ is any orthonormal basis of $L^2(\R,A_{\alpha, \beta})$. Moreover, if $\mathcal{A}$ is positive, then 
\[ tr(\mathcal{A})=\Vert \mathcal{A} \Vert_{S_1} .\]
A compact operator $\mathcal{A}$ on the Hilbert space $L^2(\R,A_{\alpha, \beta})$ is called the Hilbert--Schmidt operator, if the positive operator $\mathcal{A}^*  \mathcal{A}$ is in the trace class $S_1$. Then for any orthonormal basis $\{v_n \}_n$ of $L^2(\R,A_{\alpha, \beta})$, we have 
\[ \Vert \mathcal{A}\Vert_{HS}^2 :=\Vert \mathcal{A}\Vert_{S_2}^2= \Vert \mathcal{A}^* \mathcal{A} \Vert_{S_1}= tr(\mathcal{A}^* \mathcal{A}) = \sum_{n=1}^\infty \Vert \mathcal{A} v_n\Vert^2_{L^2(\R,A_{\alpha, \beta})}. \]

\subsection{Boundedness and compactness of localization operators}

In this subsection, we consider window functions $g_1, g_2 \in  M^1(\R,A_{\alpha, \beta})$ and establish the following boundedness and compactness results of localization operators. 

\begin{proposition}\label{pro01}
Let $\varsigma \in L^\infty(\R^2,A_{\alpha, \beta} \otimes \sigma_{\alpha, \beta})$ and $g_1, g_2 \in  M^1(\R,A_{\alpha, \beta})$. Then the localization operator $\mathfrak{L}_{g_1, g_2}(\varsigma)$ is in $\mathcal{B}(L^2(\R,A_{\alpha, \beta}))$ and we have
\[ \left\Vert \mathfrak{L}_{g_1, g_2}(\varsigma) \right\Vert_{\mathcal{B}(L^2(\R,A_{\alpha, \beta}))} \leq \Vert \varsigma \Vert_{L^\infty(\R^2,A_{\alpha, \beta} \otimes \sigma_{\alpha, \beta})} \; \Vert g_1 \Vert_{M^1(\R,A_{\alpha, \beta})} \; \Vert g_2 \Vert_{M^1(\R,A_{\alpha, \beta})}. \]
\end{proposition}
\begin{proof}
For every $f, h \in L^2(\R,A_{\alpha, \beta})$, using
H\"older's inequality,  we obtain
\begin{eqnarray*}
&& \left| \left\langle \mathfrak{L}_{g_1, g_2}(\varsigma)(f), h \right\rangle_{L^2(\R,A_{\alpha, \beta})} \right|\\ 
&& \leq \iint_{\R^2} |\varsigma(x, \xi)| \; \left|\W^{(\alpha, \beta)}_{g_1}(f)(x, \xi) \right|\; \left|\W^{(\alpha, \beta)}_{g_2}(h)(x, \xi) \right| \; d(A_{\alpha, \beta} \otimes \sigma_{\alpha, \beta}) (x, \xi) \\ 
&& \leq \left\Vert \varsigma \right\Vert_{L^\infty(\R^2,A_{\alpha, \beta} \otimes \sigma_{\alpha, \beta})} \left\Vert \W^{(\alpha, \beta)}_{g_1}(f)  \right\Vert_{L^2(\R^2,A_{\alpha, \beta} \otimes \sigma_{\alpha, \beta})} \left\Vert \W^{(\alpha, \beta)}_{g_2}(h) \right\Vert_{L^2(\R^2,A_{\alpha, \beta} \otimes \sigma_{\alpha, \beta})}.
\end{eqnarray*}
Using Plancherel's formula (\ref{eq17}), we get 
\begin{eqnarray*}
&& \left| \left\langle \mathfrak{L}_{g_1, g_2}(\varsigma)(f), h \right\rangle_{L^2(\R,A_{\alpha, \beta})} \right| \\
&& \leq  \left\Vert \varsigma \right\Vert_{L^\infty(\R^2,A_{\alpha, \beta} \otimes \sigma_{\alpha, \beta})} \Vert f \Vert_{L^2(\R,A_{\alpha, \beta})} \; \Vert g_1 \Vert_{L^2(\R,A_{\alpha, \beta})} \; \Vert h  \Vert_{L^2(\R,A_{\alpha, \beta})} \; \Vert g_2 \Vert_{L^2(\R,A_{\alpha, \beta})}.
\end{eqnarray*}
Since $M^1(\R,A_{\alpha, \beta}) \subset L^2(\R,A_{\alpha, \beta})$, we have 
\[\Vert g_1 \Vert_{L^2(\R,A_{\alpha, \beta})} \leq \Vert g_1 \Vert_{M^1(\R,A_{\alpha, \beta})} \quad \text{and} \quad  \Vert g_2  \Vert_{L^2(\R,A_{\alpha, \beta})} \leq \Vert g_2 \Vert_{M^1(\R,A_{\alpha, \beta})}. \]
Hence,
\[ \left\Vert \mathfrak{L}_{g_1, g_2}(\varsigma) \right\Vert_{\mathcal{B}(L^2(\R,A_{\alpha, \beta}))} \leq \Vert \varsigma \Vert_{L^\infty(\R^2,A_{\alpha, \beta} \otimes \sigma_{\alpha, \beta})}\; \Vert g_1 \Vert_{M^1(\R,A_{\alpha, \beta})} \; \Vert g_2 \Vert_{M^1(\R,A_{\alpha, \beta})}. \]
\end{proof}

\begin{proposition}\label{pro1}
Let $\varsigma \in M^1(\R^2,A_{\alpha, \beta} \otimes \sigma_{\alpha, \beta})$ and $g_1, g_2 \in  M^1(\R,A_{\alpha, \beta})$. Then the localization operator $\mathfrak{L}_{g_1, g_2}(\varsigma)$ is in $\mathcal{B}(L^2(\R,A_{\alpha, \beta}))$ and we have 
$$\left\Vert \mathfrak{L}_{g_1, g_2}(\varsigma) \right\Vert_{\mathcal{B}(L^2(\R,A_{\alpha, \beta}))} \leq \Vert \varsigma \Vert_{M^1(\R^2,A_{\alpha, \beta} \otimes \sigma_{\alpha, \beta})} \; \Vert g_1 \Vert_{M^1(\R,A_{\alpha, \beta})} \; \Vert g_2 \Vert_{M^1(\R,A_{\alpha, \beta})}. $$
\end{proposition}
\begin{proof}
Let $f , h \in L^2(\R,A_{\alpha, \beta})$. Since $M^\infty(\R^2,A_{\alpha, \beta} \otimes \sigma_{\alpha, \beta})$ is the dual space of $M^1(\R^2,A_{\alpha, \beta} \otimes \sigma_{\alpha, \beta})$, we have 
\begin{eqnarray}\label{eq45}
&& \left| \left\langle \mathfrak{L}_{g_1, g_2}(\varsigma)(f), h \right\rangle_{L^2(\R,A_{\alpha, \beta})} \right| \nonumber \\ 
&& \leq \iint_{\R^2} |\varsigma(x, \xi)| \; \left|\W^{(\alpha, \beta)}_{g_1}(f)(x, \xi) \; \overline{\W^{(\alpha, \beta)}_{g_2}(h)(x, \xi)} \right| 
\; d(A_{\alpha, \beta} \otimes \sigma_{\alpha, \beta}) (x, \xi) \nonumber \\
&& \leq \Vert \varsigma \Vert_{M^1(\R^2,A_{\alpha, \beta} \otimes \sigma_{\alpha, \beta})} \left\Vert \W^{(\alpha, \beta)}_{g_1}(f) \cdot \overline{\W^{(\alpha, \beta)}_{g_2}(h)} \right\Vert_{M^\infty(\R^2,A_{\alpha, \beta} \otimes \sigma_{\alpha, \beta})}.
\end{eqnarray}
Since the definition of $M^\infty(\R^2,A_{\alpha, \beta} \otimes \sigma_{\alpha, \beta})$ is independent of the choice of the window $g$, we estimate the STFT of $\W^{(\alpha, \beta)}_{g_1}(f) \cdot \overline{\W^{(\alpha, \beta)}_{g_2}(h)}$ with respect to some $g \in \mathcal{S}(\R^2) \setminus \{0\}$ with $\Vert g \Vert_{L^1(\R^2,A_{\alpha, \beta} \otimes \sigma_{\alpha, \beta})} \leq 1$.  Also, we can write the STFT as the convolution $V_g f(x, \xi)=e^{-2 \pi i x \cdot \xi}(f*M_\xi g^*)(x)$, where $^*$ is the involution $g^*(x)=\overline{g(-x)}$. Since $M^1(\R^2,A_{\alpha, \beta} \otimes \sigma_{\alpha, \beta}) \subset M^\infty(\R^2,A_{\alpha, \beta} \otimes \sigma_{\alpha, \beta})$, using Young's convolution inequality, H\"older's inequality, and Plancherel's formula (\ref{eq17}), we obtain 
\begin{eqnarray}\label{eq46}
&& \left\Vert \W^{(\alpha, \beta)}_{g_1}(f) \cdot \overline{\W^{(\alpha, \beta)}_{g_2}(h)} \right\Vert_{M^\infty(\R^2,A_{\alpha, \beta} \otimes \sigma_{\alpha, \beta})} \nonumber \\
&& \leq \left\Vert \W^{(\alpha, \beta)}_{g_1}(f) \cdot \overline{\W^{(\alpha, \beta)}_{g_2}(h)} \right\Vert_{M^1(\R^2,A_{\alpha, \beta} \otimes \sigma_{\alpha, \beta})} \nonumber \\
&& = \left\Vert V_g \left( \W^{(\alpha, \beta)}_{g_1}(f) \cdot \overline{\W^{(\alpha, \beta)}_{g_2}(h)} \right) \right\Vert_{L^1(\R^4,A_{\alpha, \beta} \otimes \sigma_{\alpha, \beta})} \nonumber \\
&& = \left\Vert \left( \W^{(\alpha, \beta)}_{g_1}(f) \cdot \overline{\W^{(\alpha, \beta)}_{g_2}(h)} \right) * M_{\xi} g^*  \right\Vert_{L^1(\R^4,A_{\alpha, \beta} \otimes \sigma_{\alpha, \beta})} \nonumber \\
&& \leq \left\Vert \W^{(\alpha, \beta)}_{g_1}(f) \cdot \overline{\W^{(\alpha, \beta)}_{g_2}(h)} \right\Vert_{L^1(\R^2,A_{\alpha, \beta} \otimes \sigma_{\alpha, \beta})} \left\Vert M_{\xi} g^*  \right\Vert_{L^1(\R^2,A_{\alpha, \beta} \otimes \sigma_{\alpha, \beta})} \nonumber \\
&& \leq \left\Vert \W^{(\alpha, \beta)}_{g_1}(f) \right\Vert_{L^2(\R^2,A_{\alpha, \beta} \otimes \sigma_{\alpha, \beta})} \left\Vert \W^{(\alpha, \beta)}_{g_2}(h) \right\Vert_{L^2(\R^2,A_{\alpha, \beta} \otimes \sigma_{\alpha, \beta})} \nonumber \\
&& = \Vert f \Vert_{L^2(\R,A_{\alpha, \beta})} \; \Vert g_1 \Vert_{L^2(\R,A_{\alpha, \beta})} \; \Vert h \Vert_{L^2(\R,A_{\alpha, \beta})} \; \Vert g_2 \Vert_{L^2(\R,A_{\alpha, \beta})} \nonumber \\
&& \leq \Vert f \Vert_{L^2(\R,A_{\alpha, \beta})} \; \Vert h \Vert_{L^2(\R,A_{\alpha, \beta})} \; \Vert g_1 \Vert_{M^1(\R,A_{\alpha, \beta})} \; \Vert g_2 \Vert_{M^1(\R,A_{\alpha, \beta})}.
\end{eqnarray}
Thus from (\ref{eq45}) and (\ref{eq46}), we get  
\[\left\Vert \mathfrak{L}_{g_1, g_2}(\varsigma) \right\Vert_{\mathcal{B}(L^2(\R,A_{\alpha, \beta}))} \leq \Vert \varsigma \Vert_{M^1(\R^2,A_{\alpha, \beta} \otimes \sigma_{\alpha, \beta})} \;  \Vert g_1 \Vert_{M^1(\R,A_{\alpha, \beta})} \; \Vert g_2 \Vert_{M^1(\R,A_{\alpha, \beta})} .\]
\end{proof}

\begin{proposition}\label{pro2}
Let $\varsigma \in M^\infty(\R^2,A_{\alpha, \beta} \otimes \sigma_{\alpha, \beta})$ and $g_1, g_2 \in  M^1(\R,A_{\alpha, \beta})$. Then the localization operator $\mathfrak{L}_{g_1, g_2}(\varsigma)$ is in $\mathcal{B}(L^2(\R,A_{\alpha, \beta}))$ and we have
\[ \left\Vert \mathfrak{L}_{g_1, g_2}(\varsigma) \right\Vert_{\mathcal{B}(L^2(\R,A_{\alpha, \beta}))} \leq \Vert \varsigma \Vert_{M^\infty(\R^2,A_{\alpha, \beta} \otimes \sigma_{\alpha, \beta})} \; \Vert g_1 \Vert_{M^1(\R,A_{\alpha, \beta})} \; \Vert g_2 \Vert_{M^1(\R,A_{\alpha, \beta})}. \]
\end{proposition}    
\begin{proof}
For every $f, h \in L^2(\R,A_{\alpha, \beta})$, using the duality between $M^1(\R^2,A_{\alpha, \beta} \otimes \sigma_{\alpha, \beta})$ and $M^\infty(\R^2,A_{\alpha, \beta} \otimes \sigma_{\alpha, \beta})$, we deduce that
\begin{eqnarray*}
&& \left| \left\langle \mathfrak{L}_{g_1, g_2}(\varsigma)(f), h \right\rangle_{L^2(\R,A_{\alpha, \beta})} \right|  \\ 
&& \leq \iint_{\R^2} |\varsigma(x, \xi)| \; \left|\W^{(\alpha, \beta)}_{g_1}(f)(x, \xi) \; \overline{\W^{(\alpha, \beta)}_{g_2}(h)(x, \xi)} \right| \; d(A_{\alpha, \beta} \otimes \sigma_{\alpha, \beta}) (x, \xi) \\ 
&& \leq \left\Vert \varsigma \right\Vert_{M^\infty(\R^2,A_{\alpha, \beta} \otimes \sigma_{\alpha, \beta})} \left\Vert \W^{(\alpha, \beta)}_{g_1}(f)  \cdot \overline{\W^{(\alpha, \beta)}_{g_2}(h)} \right\Vert_{M^1(\R^2,A_{\alpha, \beta} \otimes \sigma_{\alpha, \beta})}.
\end{eqnarray*}
Now, using the estimate obtained in (\ref{eq46}), we get 
\begin{eqnarray*}
&& \left| \left\langle \mathfrak{L}_{g_1, g_2}(\varsigma)(f), h \right\rangle_{L^2(\R,A_{\alpha, \beta})} \right| \\
&& \leq \left\Vert \varsigma \right\Vert_{M^\infty(\R^2,A_{\alpha, \beta} \otimes \sigma_{\alpha, \beta})} \Vert f \Vert_{L^2(\R,A_{\alpha, \beta})} \; \Vert h \Vert_{L^2(\R,A_{\alpha, \beta})} \; \Vert g_1 \Vert_{M^1(\R,A_{\alpha, \beta})} \; \Vert g_2 \Vert_{M^1(\R,A_{\alpha, \beta})}.
\end{eqnarray*}
Hence,
\[ \left\Vert \mathfrak{L}_{g_1, g_2}(\varsigma) \right\Vert_{\mathcal{B}(L^2(\R,A_{\alpha, \beta}))} \leq \Vert \varsigma \Vert_{M^\infty(\R^2,A_{\alpha, \beta} \otimes \sigma_{\alpha, \beta})} \; \Vert g_1 \Vert_{M^1(\R,A_{\alpha, \beta})} \; \Vert g_2 \Vert_{M^1(\R,A_{\alpha, \beta})}. \]
\end{proof}

\begin{theorem}\label{th1}
Let $\varsigma \in M^p(\R^2,A_{\alpha, \beta} \otimes \sigma_{\alpha, \beta})$, $1 < p < \infty$ and $g_1, g_2 \in  M^1(\R,A_{\alpha, \beta})$. Then, for fixed $\varsigma \in M^p(\R^2,A_{\alpha, \beta} \otimes \sigma_{\alpha, \beta})$ the operator $\mathfrak{L}_{g_1, g_2}$ can be uniquely extended to a bounded linear operator on $L^2(\R,A_{\alpha, \beta})$, such that
\[ \left\Vert \mathfrak{L}_{g_1, g_2}(\varsigma) \right\Vert_{\mathcal{B}(L^2(\R,A_{\alpha, \beta}))} \leq \Vert \varsigma \Vert_{M^p(\R^2,A_{\alpha, \beta} \otimes \sigma_{\alpha, \beta})} \; \Vert g_1 \Vert_{M^1(\R,A_{\alpha, \beta})} \; \Vert g_2 \Vert_{M^1(\R,A_{\alpha, \beta})}. \] 
\end{theorem}
\begin{proof}
Let $1< p < \infty$. For every $\varsigma \in M^1(\R^2,A_{\alpha, \beta} \otimes \sigma_{\alpha, \beta}) \cap  M^\infty(\R^2,A_{\alpha, \beta} \otimes \sigma_{\alpha, \beta})$, using Proposition \ref{pro1}, Proposition \ref{pro2}, the fact that the modulation spaces $M^{p,q}$ interpolate exactly like the corresponding mixed-norm spaces $L^{p,q}$ and the Riesz--Thorin interpolation theorem (see \cite{ste56}), we obtain 
\[ \left\Vert \mathfrak{L}_{g_1, g_2}(\varsigma) \right\Vert_{\mathcal{B}(L^2(\R,A_{\alpha, \beta}))} \leq \Vert \varsigma \Vert_{M^p(\R^2,A_{\alpha, \beta} \otimes \sigma_{\alpha, \beta})} \; \Vert g_1 \Vert_{M^1(\R,A_{\alpha, \beta})} \; \Vert g_2 \Vert_{M^1(\R,A_{\alpha, \beta})} . \]
Let $\varsigma \in M^p(\R^2,A_{\alpha, \beta} \otimes \sigma_{\alpha, \beta})$ and $\{ \varsigma_n \}_{n \geq 1}$ be a sequence of functions in $M^1(\R^2,A_{\alpha, \beta} \otimes \sigma_{\alpha, \beta}) \cap M^\infty(\R^2,A_{\alpha, \beta} \otimes \sigma_{\alpha, \beta})$ such that $\varsigma_n \to \varsigma$ in $M^p(\R^2,A_{\alpha, \beta} \otimes \sigma_{\alpha, \beta})$ as $ n \to \infty$. Hence for every $n, k \in \mathbb{N}$, we have 
\[ \left\Vert \mathfrak{L}_{g_1, g_2}(\varsigma_n) - \mathfrak{L}_{g_1, g_2}(\varsigma_k) \right\Vert_{\mathcal{B}(L^2(\R,A_{\alpha, \beta}))} \leq \Vert \varsigma_n - \varsigma_k \Vert_{M^p(\R^2,A_{\alpha, \beta} \otimes \sigma_{\alpha, \beta})} \; \Vert g_1 \Vert_{M^1(\R,A_{\alpha, \beta})} \; \Vert g_2 \Vert_{M^1(\R,A_{\alpha, \beta})} . \]
Therefore, $\{ \mathfrak{L}_{g_1, g_2}(\varsigma_n) \}_{n \geq 1}$ is a Cauchy sequence in $\mathcal{B}(L^2(\R,A_{\alpha, \beta}))$. Let $\{ \mathfrak{L}_{g_1, g_2}(\varsigma_n) \}_{n \geq 1}$ converges to $ \mathfrak{L}_{g_1, g_2}(\varsigma)$. Then the limit $ \mathfrak{L}_{g_1, g_2}(\varsigma)$ is independent of the choice of $\{ \varsigma_n \}_{n \geq 1}$ and we obtain
\begin{eqnarray*}
\left\Vert \mathfrak{L}_{g_1, g_2}(\varsigma) \right\Vert_{\mathcal{B}(L^2(\R,A_{\alpha, \beta}))} 
& = & \lim_{n \to \infty} \left\Vert \mathfrak{L}_{g_1, g_2}(\varsigma_n) \right\Vert_{\mathcal{B}(L^2(\R,A_{\alpha, \beta}))} \\
& \leq & \lim_{n \to \infty} \Vert \varsigma_n \Vert_{M^p(\R^2,A_{\alpha, \beta} \otimes \sigma_{\alpha, \beta})} \; \Vert g_1 \Vert_{M^1(\R,A_{\alpha, \beta})} \; \Vert g_2 \Vert_{M^1(\R,A_{\alpha, \beta})} \\
& =& \Vert \varsigma \Vert_{M^p(\R^2,A_{\alpha, \beta} \otimes \sigma_{\alpha, \beta})} \; \Vert g_1 \Vert_{M^1(\R,A_{\alpha, \beta})} \; \Vert g_2 \Vert_{M^1(\R,A_{\alpha, \beta})}.
\end{eqnarray*}
This completes the proof. 
\end{proof}

\begin{theorem}\label{th2}
Let $\varsigma \in M^p(\R^2,A_{\alpha, \beta} \otimes \sigma_{\alpha, \beta})$, $1 \leq p < \infty$ and $g_1, g_2 \in  M^1(\R,A_{\alpha, \beta})$. Then the localization operator $\mathfrak{L}_{g_1, g_2}(\varsigma): L^2(\R,A_{\alpha, \beta}) \to L^2(\R,A_{\alpha, \beta}) $ is compact. 
\end{theorem}
\begin{proof}
Assume that $\varsigma \in M^1(\R^2,A_{\alpha, \beta} \otimes \sigma_{\alpha, \beta})$. Let $\{ v_n \}_n$ be an orthonormal basis for $L^2(\R,A_{\alpha, \beta})$. Since $M^1(\R^2,A_{\alpha, \beta} \otimes \sigma_{\alpha, \beta}) \subset L^1(\R^2,A_{\alpha, \beta} \otimes \sigma_{\alpha, \beta})$, using Fubini's theorem and Parseval's identity, we obtain 
\begin{eqnarray*}
&& \sum_{n=1}^\infty \left\langle \mathfrak{L}_{g_1, g_2}(\varsigma)(v_n), v_n \right\rangle_{L^2(\R,A_{\alpha, \beta})} \\
&& = \sum_{n=1}^\infty  \iint_{\R^2} \varsigma (x, \xi)  \langle v_n, {g_1}_{x, \xi}^{(\alpha, \beta)} (- \; \cdot) \rangle_{L^2(\R,A_{\alpha, \beta})} \langle {g_2}_{x, \xi}^{(\alpha, \beta)} (- \; \cdot), v_n \rangle_{L^2(\R,A_{\alpha, \beta})} d(A_{\alpha, \beta} \otimes \sigma_{\alpha, \beta}) (x, \xi)\\
&& \leq \iint_{\R^2} \varsigma (x, \xi)  \left( \sum_{n=1}^\infty \langle v_n, {g_1}_{x, \xi}^{(\alpha, \beta)} (- \; \cdot) \rangle_{L^2(\R,A_{\alpha, \beta})}  \langle {g_2}_{x, \xi}^{(\alpha, \beta)} (- \; \cdot), v_n \rangle_{L^2(\R,A_{\alpha, \beta})} \right) \\
&& \quad \times \; d(A_{\alpha, \beta} \otimes \sigma_{\alpha, \beta}) (x, \xi) \\
&& \leq \frac{1}{2} \iint_{\R^2} \varsigma (x, \xi)  \left( \sum_{n=1}^\infty |\langle v_n, {g_1}_{x, \xi}^{(\alpha, \beta)} (- \; \cdot) \rangle_{L^2(\R,A_{\alpha, \beta})}|^2 + \sum_{n=1}^\infty |\langle {g_2}_{x, \xi}^{(\alpha, \beta)} (- \; \cdot), v_n \rangle_{L^2(\R,A_{\alpha, \beta})}|^2 \right) \\
&& \quad \times \; d(A_{\alpha, \beta} \otimes \sigma_{\alpha, \beta}) (x, \xi)\\
&& = \frac{1}{2} \Vert \varsigma \Vert_{L^1(\R^2,A_{\alpha, \beta} \otimes \sigma_{\alpha, \beta})} \; (\Vert g_1 \Vert^2_{L^2(\R,A_{\alpha, \beta})} + \Vert g_2 \Vert^2_{L^2(\R,A_{\alpha, \beta})}) \\
\end{eqnarray*}
\begin{eqnarray*}
&& \leq \frac{1}{2} \Vert \varsigma \Vert_{M^1(\R^2,A_{\alpha, \beta} \otimes \sigma_{\alpha, \beta})} \; (\Vert g_1 \Vert^2_{M^1(\R,A_{\alpha, \beta})} + \Vert g_2 \Vert^2_{M^1(\R,A_{\alpha, \beta})}). \qquad \qquad \qquad \qquad \qquad \qquad \qquad \qquad \qquad \qquad \qquad
\end{eqnarray*}
Therefore, the operator $\mathfrak{L}_{g_1, g_2}(\varsigma)$ is in $S_1$. Next, assume that $\varsigma \in M^p(\R^2,A_{\alpha, \beta} \otimes \sigma_{\alpha, \beta})$. We consider a sequence of functions $\{ \varsigma_n \}_{n \geq 1}$ in $M^1(\R^2,A_{\alpha, \beta} \otimes \sigma_{\alpha, \beta}) \cap M^\infty(\R^2,A_{\alpha, \beta} \otimes \sigma_{\alpha, \beta})$ such that $\varsigma_n \to \varsigma$ in $M^p(\R^2,A_{\alpha, \beta} \otimes \sigma_{\alpha, \beta})$ as $n \to \infty$. Then, using Theorem \ref{th1}, we get
\[ \Vert \mathfrak{L}_{g_1, g_2}(\varsigma_n) - \mathfrak{L}_{g_1, g_2}(\varsigma)  \Vert_{\mathcal{B}(L^2(\R,A_{\alpha, \beta}))} \leq \Vert \varsigma_n - \varsigma \Vert_{M^p(\R^2,A_{\alpha, \beta} \otimes \sigma_{\alpha, \beta})} \; \Vert g_1 \Vert_{M^1(\R,A_{\alpha, \beta})} \; \Vert g_2 \Vert_{M^1(\R,A_{\alpha, \beta})}  \to 0, \]
as $n \to \infty$. Hence, $\mathfrak{L}_{g_1, g_2}(\varsigma_n) \to \mathfrak{L}_{g_1, g_2}(\varsigma) $ in $\mathcal{B}(L^2(\R,A_{\alpha, \beta}))$ as $n \to \infty$. From the above, we obtain that $ \{ \mathfrak{L}_{g_1, g_2}(\varsigma_n)\}_{n \geq 1}$ is a sequence of linear operators in $S_1$ and hence compact, so $\mathfrak{L}_{g_1, g_2}(\varsigma)$ is compact. 
\end{proof}

\subsection{Boundedness of $\mathfrak{L}_{g_1, g_2}(\varsigma)$ on $S_p$}

In this subsection, we prove that the localization operator $\mathfrak{L}_{g_1, g_2}(\varsigma)$ is in $S_p$ and provide an upper bound of the norm $\Vert \mathfrak{L}_{g_1, g_2}(\varsigma) \Vert_{S_p}$. 
We begin with the following proposition.

\begin{proposition}\label{pro3}
Let $\varsigma \in  M^1(\R^2,A_{\alpha, \beta} \otimes \sigma_{\alpha, \beta})$ and $g_1, g_2 \in  M^1(\R,A_{\alpha, \beta})$. Then the localization operator $\mathfrak{L}_{g_1, g_2}(\varsigma): L^2(\R,A_{\alpha, \beta}) \to L^2(\R,A_{\alpha, \beta}) $ is in $S_1$ and 
\[ \Vert \mathfrak{L}_{g_1, g_2}(\varsigma) \Vert_{S_1} \leq 2 \Vert \varsigma \Vert_{M^1(\R^2,A_{\alpha, \beta} \otimes \sigma_{\alpha, \beta})} \; ( \Vert g_1 \Vert^2_{M^1(\R,A_{\alpha, \beta})} + \Vert g_2 \Vert^2_{M^1(\R,A_{\alpha, \beta})}). \]
\end{proposition}
\begin{proof}
If $\varsigma \in  M^1(\R^2,A_{\alpha, \beta} \otimes \sigma_{\alpha, \beta})$, then from the first part of the proof of Theorem \ref{th2}, the operator $\mathfrak{L}_{g_1, g_2}(\varsigma)$ is in $S_1$. Now, to prove the estimate, assume that $\varsigma$ is non-negative and in $M^1(\R^2,A_{\alpha, \beta} \otimes \sigma_{\alpha, \beta})$. Then $(\mathfrak{L}_{g_1, g_2}(\varsigma)^* \mathfrak{L}_{g_1, g_2}(\varsigma))^{1/2}=\mathfrak{L}_{g_1, g_2}(\varsigma)$ (see \cite{bac2016, won02}). Let $\{v_n\}_n$ be an orthonormal basis for $L^2(\R,A_{\alpha, \beta})$ consisting of eigenvalues of $(\mathfrak{L}_{g_1, g_2}(\varsigma)^* \mathfrak{L}_{g_1, g_2}(\varsigma))^{1/2}: L^2(\R,A_{\alpha, \beta}) \to L^2(\R,A_{\alpha, \beta})$. 
Then by using the estimate obtained in the first part of the proof of Theorem \ref{th2}, we get
\begin{eqnarray}\label{eq43}
\Vert \mathfrak{L}_{g_1, g_2}(\varsigma) \Vert_{S_1}
& =& \sum_{n=1}^\infty \left\langle (\mathfrak{L}_{g_1, g_2}(\varsigma)^* \mathfrak{L}_{g_1, g_2}(\varsigma))^{1/2}(v_n), v_n \right\rangle_{L^2(\R,A_{\alpha, \beta})} \nonumber \\
&=& \sum_{n=1}^\infty \left\langle \mathfrak{L}_{g_1, g_2}(\varsigma)(v_n), v_n \right\rangle_{L^2(\R,A_{\alpha, \beta})} \nonumber \\
& \leq & \frac{1}{2} \Vert \varsigma \Vert_{M^1(\R^2,A_{\alpha, \beta} \otimes \sigma_{\alpha, \beta})} \; ( \Vert g_1 \Vert^2_{M^1(\R,A_{\alpha, \beta})} + \Vert g_2 \Vert^2_{M^1(\R,A_{\alpha, \beta})}).
\end{eqnarray}
Next, assume that $\varsigma \in M^1(\R^2,A_{\alpha, \beta} \otimes \sigma_{\alpha, \beta})$ is  an arbitrary real-valued function. We can write $\varsigma = \varsigma_+ - \varsigma_-$, where $\varsigma_+=\max(\varsigma, 0)$ and $\varsigma_-=-\min(\varsigma, 0)$. Then, using relation (\ref{eq43}), we obtain
\begin{eqnarray}\label{eq44}
\Vert \mathfrak{L}_{g_1, g_2}(\varsigma) \Vert_{S_1}
& =& \Vert \mathfrak{L}_{g_1, g_2}(\varsigma_+) -\mathfrak{L}_{g_1, g_2}(\varsigma_-) \Vert_{S_1} \nonumber \\
& \leq & \Vert \mathfrak{L}_{g_1, g_2}(\varsigma_+) \Vert_{S_1} +\Vert \mathfrak{L}_{g_1, g_2}(\varsigma_-) \Vert_{S_1} \nonumber \\
& \leq & \frac{1}{2} ( \Vert g_1 \Vert^2_{M^1(\R,A_{\alpha, \beta})} + \Vert g_2 \Vert^2_{M^1(\R,A_{\alpha, \beta})})  (\Vert \varsigma_+ \Vert_{M^1(\R^2,A_{\alpha, \beta} \otimes \sigma_{\alpha, \beta})} + \Vert \varsigma_- \Vert_{M^1(\R^2,A_{\alpha, \beta} \otimes \sigma_{\alpha, \beta})}) \nonumber \\
& \leq & ( \Vert g_1 \Vert^2_{M^1(\R,A_{\alpha, \beta})} + \Vert g_2 \Vert^2_{M^1(\R,A_{\alpha, \beta})}) \Vert \varsigma \Vert_{M^1(\R^2,A_{\alpha, \beta} \otimes \sigma_{\alpha, \beta})}.
\end{eqnarray}
Finally, assume that $\varsigma \in  M^1(\R^2,A_{\alpha, \beta} \otimes \sigma_{\alpha, \beta})$ is a complex-valued function. Then, we can write $\varsigma=\varsigma_1+ i \varsigma_2$, where $\varsigma_1, \varsigma_2$ are the real and imaginary parts of $\varsigma$ respectively. Then, using relation (\ref{eq44}), we obtain 
\begin{eqnarray*}
\Vert \mathfrak{L}_{g_1, g_2}(\varsigma) \Vert_{S_1}
& =& \Vert \mathfrak{L}_{g_1, g_2}(\varsigma_1)+i \mathfrak{L}_{g_1, g_2}(\varsigma_2)  \Vert_{S_1} \\
& \leq & \Vert \mathfrak{L}_{g_1, g_2}(\varsigma_1) \Vert_{S_1} + \Vert \mathfrak{L}_{g_1, g_2}(\varsigma_2) \Vert_{S_1} \\
&\leq & ( \Vert g_1 \Vert^2_{M^1(\R,A_{\alpha, \beta})} + \Vert g_2 \Vert^2_{M^1(\R,A_{\alpha, \beta})}) (\Vert \varsigma_1 \Vert_{M^1(\R^2,A_{\alpha, \beta} \otimes \sigma_{\alpha, \beta})} + \Vert \varsigma_2 \Vert_{M^1(\R^2,A_{\alpha, \beta} \otimes \sigma_{\alpha, \beta})}) \nonumber \\
& \leq & 2 \Vert \varsigma \Vert_{M^1(\R^2,A_{\alpha, \beta} \otimes \sigma_{\alpha, \beta})} ( \Vert g_1 \Vert^2_{M^1(\R,A_{\alpha, \beta})} + \Vert g_2 \Vert^2_{M^1(\R,A_{\alpha, \beta})}) .
\end{eqnarray*}
This completes the proof.
\end{proof}

\begin{theorem}
Let $\varsigma \in  M^p(\R^2,A_{\alpha, \beta} \otimes \sigma_{\alpha, \beta})$, $1 \leq p \leq \infty$ and $g_1, g_2 \in  M^1(\R,A_{\alpha, \beta})$. Then the localization operator $\mathfrak{L}_{g_1, g_2}(\varsigma): L^2(\R,A_{\alpha, \beta}) \to L^2(\R,A_{\alpha, \beta})$ is in $S_p$ and 
\[\Vert \mathfrak{L}_{g_1, g_2}(\varsigma) \Vert_{S_p} \leq 2^{1/p} \Vert \varsigma  \Vert_{M^p(\R^2,A_{\alpha, \beta} \otimes \sigma_{\alpha, \beta})} \; ( \Vert g_1 \Vert^2_{M^1(\R,A_{\alpha, \beta})} + \Vert g_2 \Vert^2_{M^1(\R,A_{\alpha, \beta})})^{1/p}.\]
\end{theorem}
\begin{proof}
The proof follows from Proposition \ref{pro2}, Proposition \ref{pro3} and by interpolation theorems (see \cite{won02}, Theorems 2.10 and 2.11). 
\end{proof}

Next, we improve the constant given in the previous proposition and also we provide a lower bound of the norm $\Vert \mathfrak{L}_{g_1, g_2}(\varsigma) \Vert_{S_1}$, more precisely we have the following.

\begin{theorem}\label{th4}
Let $\varsigma \in M^1(\R^2,A_{\alpha, \beta} \otimes \sigma_{\alpha, \beta})$ and $g_1, g_2 \in  M^1(\R,A_{\alpha, \beta})$. Then the localization operator $\mathfrak{L}_{g_1, g_2}(\varsigma)$ is in $S_1$ and we have
\begin{eqnarray*}
&& \frac{2}{\Vert g_1 \Vert^2_{M^1(\R,A_{\alpha, \beta})} + \Vert g_2 \Vert^2_{M^1(\R,A_{\alpha, \beta})}}\; \Vert \tilde{\varsigma} \Vert_{M^1(\R^2,A_{\alpha, \beta} \otimes \sigma_{\alpha, \beta})} \\
&& \leq \Vert \mathfrak{L}_{g_1, g_2}(\varsigma) \Vert_{S_1} \leq  \frac{1}{2} (\Vert g_1 \Vert^2_{M^1(\R,A_{\alpha, \beta})} + \Vert g_2 \Vert^2_{M^1(\R,A_{\alpha, \beta})}) \; \Vert \varsigma \Vert_{M^1(\R^2,A_{\alpha, \beta} \otimes \sigma_{\alpha, \beta})},
\end{eqnarray*}
where $\tilde{\varsigma}$ is given by $\tilde{\varsigma}(x, \xi)=\left\langle \mathfrak{L}_{g_1, g_2}(\varsigma)({g_1}_{x, \xi}^{(\alpha, \beta)}), {g_2}_{x, \xi}^{(\alpha, \beta)} \right\rangle_{L^2(\R,A_{\alpha, \beta})}$.
\end{theorem}
\begin{proof}
Since $\varsigma \in M^1(\R^2,A_{\alpha, \beta} \otimes \sigma_{\alpha, \beta})$, by Proposition \ref{pro3}, $\mathfrak{L}_{g_1, g_2}(\varsigma)$ is in $S_1$. Using the canonical form of compact operators (see \cite{won02}, Theorem 2.2), we obtain 
\begin{equation}\label{eq49}
\mathfrak{L}_{g_1, g_2}(\varsigma)(f) =\sum_{n=1}^{\infty} s_n(\mathfrak{L}_{g_1, g_2}(\varsigma)) \langle f, v_n \rangle_{L^2(\R,A_{\alpha, \beta})} u_n,
\end{equation}
where $\{s_n(\mathfrak{L}_{g_1, g_2}(\varsigma)) \}_n$ are the positive singular values of $\mathfrak{L}_{g_1, g_2}(\varsigma)$, $\{ v_n \}_n$ is an orthonormal basis for the orthogonal complement of the null space of $\mathfrak{L}_{g_1, g_2}(\varsigma)$ consisting of eigenvectors of $|\mathfrak{L}_{g_1, g_2}(\varsigma)|$ and $\{ u_n \}_n$ is an orthonormal set in $L^2(\R,A_{\alpha, \beta})$. Then we have
\[ \sum_{n=1}^{\infty}  \langle \mathfrak{L}_{g_1, g_2}(\varsigma)(v_n), u_n \rangle_{L^2(\R,A_{\alpha, \beta})}   = \sum_{n=1}^{\infty} s_n(\mathfrak{L}_{g_1, g_2}(\varsigma))= \Vert \mathfrak{L}_{g_1, g_2}(\varsigma) \Vert_{S_1}.  \]
Now, using Fubini's theorem, Cauchy--Schwarz's inequality, and Bessel's inequality, we get
\begin{eqnarray*}
\Vert \mathfrak{L}_{g_1, g_2}(\varsigma) \Vert_{S_1} 
& = &  \sum_{n=1}^{\infty}  \langle \mathfrak{L}_{g_1, g_2}(\varsigma)(v_n), u_n \rangle_{L^2(\R,A_{\alpha, \beta})} \\
& = & \sum_{n=1}^{\infty} \iint_{\R^2} \varsigma(x, \xi) \; \W^{(\alpha, \beta)}_{g_1}(v_n)(x, \xi)\; \overline{\W^{(\alpha, \beta)}_{g_2}(u_n)(x, \xi)} \; d(A_{\alpha, \beta} \otimes \sigma_{\alpha, \beta}) (x, \xi) \\
& \leq & \iint_{\R^2} |\varsigma(x, \xi)| \left(\sum_{n=1}^{\infty} |\W^{(\alpha, \beta)}_{g_1}(v_n)(x, \xi)|^2 \right)^{1/2} \left(\sum_{n=1}^{\infty} |\W^{(\alpha, \beta)}_{g_2}(u_n)(x, \xi)|^2 \right)^{1/2}\\
&& \;  \times \; d(A_{\alpha, \beta} \otimes \sigma_{\alpha, \beta}) (x, \xi) \\
& \leq & \Vert \varsigma \Vert_{L^1(\R^2,A_{\alpha, \beta} \otimes \sigma_{\alpha, \beta})} \;\Vert {g_1}_{x, \xi}^{(\alpha, \beta)}  \Vert_{L^2(\R,A_{\alpha, \beta})} \; \Vert {g_2}_{x, \xi}^{(\alpha, \beta)}  \Vert_{L^2(\R,A_{\alpha, \beta})}  \\
& \leq & \Vert \varsigma \Vert_{M^1(\R^2,A_{\alpha, \beta} \otimes \sigma_{\alpha, \beta})}\; \Vert g_1 \Vert_{M^1(\R,A_{\alpha, \beta})} \; \Vert g_2 \Vert_{M^1(\R,A_{\alpha, \beta})} \\
& \leq & \frac{1}{2} (\Vert g_1 \Vert^2_{M^1(\R,A_{\alpha, \beta})} + \Vert g_2 \Vert^2_{M^1(\R,A_{\alpha, \beta})}) \; \Vert \varsigma \Vert_{M^1(\R^2,A_{\alpha, \beta} \otimes \sigma_{\alpha, \beta})}.
\end{eqnarray*}
Next, we show that $\tilde{\varsigma} \in M^1(\R^2,A_{\alpha, \beta} \otimes \sigma_{\alpha, \beta})$.  Using formula (\ref{eq49}), we obtain 
\begin{eqnarray*}
|\tilde{\varsigma}(x, \xi)|
& = & \left|\left\langle \mathfrak{L}_{g_1, g_2}(\varsigma)({g_1}_{x, \xi}^{(\alpha, \beta)}), {g_2}_{x, \xi}^{(\alpha, \beta)} \right\rangle_{L^2(\R,A_{\alpha, \beta})}\right| \\
& = & \left| \sum_{n=1}^{\infty} s_n(\mathfrak{L}_{g_1, g_2}(\varsigma)) \left\langle {g_1}_{x, \xi}^{(\alpha, \beta)}, v_n \right\rangle_{L^2(\R,A_{\alpha, \beta})} \left\langle u_n, {g_2}_{x, \xi}^{(\alpha, \beta)} \right\rangle_{L^2(\R,A_{\alpha, \beta})} \right| \\
& \leq & \frac{1}{2} \sum_{n=1}^{\infty} s_n(\mathfrak{L}_{g_1, g_2}(\varsigma)) \left( \left| \left\langle {g_1}_{x, \xi}^{(\alpha, \beta)}, v_n \right\rangle_{L^2(\R,A_{\alpha, \beta})}  \right|^2 + \left| \left\langle {g_2}_{x, \xi}^{(\alpha, \beta)}, u_n  \right\rangle_{L^2(\R,A_{\alpha, \beta})} \right|^2 \right). 
\end{eqnarray*}  
Now, using Plancherel's formula (\ref{eq17}) and Fubini's theorem, we get 
\begin{eqnarray}\label{eq50}
\Vert \tilde{\varsigma} \Vert_{L^1(\R^2,A_{\alpha, \beta} \otimes \sigma_{\alpha, \beta})} 
& = & \iint_{\R^2} |\tilde{\varsigma}(x, \xi)| \; d(A_{\alpha, \beta} \otimes \sigma_{\alpha, \beta}) (x, \xi) \nonumber \\
& \leq & \frac{1}{2} \sum_{n=1}^{\infty} s_n(\mathfrak{L}_{g_1, g_2}(\varsigma)) \left( \iint_{\R^2} \left| \left\langle {g_1}_{x, \xi}^{(\alpha, \beta)}, v_n \right\rangle_{L^2(\R,A_{\alpha, \beta})} \right|^2 d(A_{\alpha, \beta} \otimes \sigma_{\alpha, \beta}) (x, \xi) \right. \nonumber \\
&& \;\; \left. + \iint_{\R^2} \left| \left\langle {g_2}_{x, \xi}^{(\alpha, \beta)}, u_n  \right\rangle_{L^2(\R,A_{\alpha, \beta})} \right|^2 d(A_{\alpha, \beta} \otimes \sigma_{\alpha, \beta}) (x, \xi)  \right) \nonumber \\
&\leq & \frac{1}{2} (\Vert g_1 \Vert^2_{M^1(\R,A_{\alpha, \beta})} + \Vert g_2 \Vert^2_{M^1(\R,A_{\alpha, \beta})}) \sum_{n=1}^{\infty} s_n(\mathfrak{L}_{g_1, g_2}(\varsigma)) \nonumber \\
& = & \frac{1}{2} (\Vert g_1 \Vert^2_{M^1(\R,A_{\alpha, \beta})} + \Vert g_2 \Vert^2_{M^1(\R,A_{\alpha, \beta})}) \; \Vert \mathfrak{L}_{g_1, g_2}(\varsigma) \Vert_{S_1} .
\end{eqnarray}
Also, we have
\begin{eqnarray}\label{eq51}
\Vert \tilde{\varsigma} \Vert_{M^1(\R^2,A_{\alpha, \beta} \otimes \sigma_{\alpha, \beta})} = \Vert V_g \tilde{\varsigma} \Vert_{L^1(\R^4,A_{\alpha, \beta} \otimes \sigma_{\alpha, \beta})} 
& = & \Vert \tilde{\varsigma} * M_\xi g^*  \Vert_{L^1(\R^4,A_{\alpha, \beta} \otimes \sigma_{\alpha, \beta})} \nonumber \\
& \leq & \Vert \tilde{\varsigma} \Vert_{L^1(\R^2,A_{\alpha, \beta} \otimes \sigma_{\alpha, \beta})}, 
\end{eqnarray}
where we obtain the estimate with respect to some $g \in \mathcal{S}(\R^2) \setminus \{0\}$ with $\Vert g \Vert_{L^1(\R^2,A_{\alpha, \beta} \otimes \sigma_{\alpha, \beta})} \leq 1$. Finally, using (\ref{eq50}) and (\ref{eq51}), we get 
\[\Vert \tilde{\varsigma} \Vert_{M^1(\R^2,A_{\alpha, \beta} \otimes \sigma_{\alpha, \beta})} \leq  \frac{1}{2} (\Vert g_1 \Vert^2_{M^1(\R,A_{\alpha, \beta})} + \Vert g_2 \Vert^2_{M^1(\R,A_{\alpha, \beta})}) \; \Vert \mathfrak{L}_{g_1, g_2}(\varsigma) \Vert_{S_1} .  \]  
This completes the proof of the theorem.
\end{proof}

\subsection{$M^p(\R,A_{\alpha, \beta})$ Boundedness}
 
In this section, we prove that the localization operators $\mathfrak{L}_{g_1, g_2}(\varsigma) : M^p(\R,A_{\alpha, \beta}) \to M^p(\R,A_{\alpha, \beta})$ are bounded. We begin with the following propositions. 

\begin{proposition}\label{pro4} 
Let $\varsigma \in M^1(\R^2,A_{\alpha, \beta} \otimes \sigma_{\alpha, \beta})$, $g_1 \in  M^{p'}(\R,A_{\alpha, \beta})$ and $g_2 \in  M^p(\R,A_{\alpha, \beta})$, for $1 \leq p \leq \infty$. Then the localization operator $\mathfrak{L}_{g_1, g_2}(\varsigma): M^p(\R,A_{\alpha, \beta}) \to M^p(\R,A_{\alpha, \beta})$ is a bounded linear operator, and we have
\[\Vert \mathfrak{L}_{g_1, g_2}(\varsigma) \Vert_{\mathcal{B}(M^p(\R,A_{\alpha, \beta}))} \leq  \Vert \varsigma  \Vert_{M^1(\R^2,A_{\alpha, \beta} \otimes \sigma_{\alpha, \beta})} \; \Vert g_1 \Vert_{ M^{p'}(\R,A_{\alpha, \beta})} \; \Vert g_2 \Vert_{M^p(\R,A_{\alpha, \beta})}. \] 
\end{proposition}
\begin{proof}
Let $f \in M^p(\R,A_{\alpha, \beta})$, $1 \leq p \leq \infty$ and $g \in M^{p'}(\R,A_{\alpha, \beta})$. Then from H\"older's inequality and the relation (\ref{eq14}), we have 
\begin{equation}\label{eq52}
\left|\W^{(\alpha, \beta)}_{g}(f)(x, \xi) \right| \leq \Vert f  \Vert_{M^p(\R,A_{\alpha, \beta})} \; \Vert g  \Vert_{M^{p'}(\R,A_{\alpha, \beta})}. 
\end{equation} 
For every $f \in M^p(\R,A_{\alpha, \beta})$ and $h \in M^{p'}(\R,A_{\alpha, \beta})$, using the relations (\ref{eq42}) and (\ref{eq52}), we obtain
\begin{eqnarray*}
&& \left| \left\langle \mathfrak{L}_{g_1, g_2}(\varsigma)(f), h \right\rangle \right|  \\
&& \leq \iint_{\R^2} |\varsigma(x, \xi)| \; \left|\W^{(\alpha, \beta)}_{g_1}(f)(x, \xi) \right| \left|   \W^{(\alpha, \beta)}_{g_2}(h)(x, \xi) \right| \; d(A_{\alpha, \beta} \otimes \sigma_{\alpha, \beta}) (x, \xi) \\
&& \leq \Vert \varsigma \Vert_{M^1(\R^2,A_{\alpha, \beta} \otimes \sigma_{\alpha, \beta})} \; \Vert f \Vert_{M^p(\R,A_{\alpha, \beta})} \;  \Vert g_1 \Vert_{M^{p'}(\R,A_{\alpha, \beta})} \; \Vert h \Vert_{M^{p'}(\R,A_{\alpha, \beta})} \; \Vert g_2 \Vert_{M^p(\R,A_{\alpha, \beta})}. 
\end{eqnarray*}
Hence,
\[\Vert \mathfrak{L}_{g_1, g_2}(\varsigma) \Vert_{\mathcal{B}(M^p(\R,A_{\alpha, \beta}))} \leq  \Vert \varsigma  \Vert_{M^1(\R^2,A_{\alpha, \beta} \otimes \sigma_{\alpha, \beta})} \; \Vert g_1 \Vert_{ M^{p'}(\R,A_{\alpha, \beta})} \; \Vert g_2 \Vert_{M^p(\R,A_{\alpha, \beta})}. \] 
\end{proof}

Next, we obtain an $M^p(\R,A_{\alpha, \beta})$-boundedness result using the Schur technique. The estimate obtained for the norm $\Vert \mathfrak{L}_{g_1, g_2}(\varsigma) \Vert_{\mathcal{B}(M^p(\R,A_{\alpha, \beta}))}$ is different from the previous Proposition. 

\begin{proposition}\label{pro5}
Let $\varsigma \in M^1(\R^2,A_{\alpha, \beta} \otimes \sigma_{\alpha, \beta})$ and $g_1, g_2 \in  M^1(\R,A_{\alpha, \beta}) \cap L^\infty(\R,A_{\alpha, \beta})$. Then there exists a bounded linear operator $\mathfrak{L}_{g_1, g_2}(\varsigma): M^p(\R,A_{\alpha, \beta}) \to  M^p(\R,A_{\alpha, \beta}) $, $1 \leq p \leq \infty$ such that
\begin{eqnarray*}
&& \Vert \mathfrak{L}_{g_1, g_2}(\varsigma) \Vert_{\mathcal{B}(M^p(\R,A_{\alpha, \beta}))} \\
&& \leq \max(\Vert g_1 \Vert_{M^1(\R,A_{\alpha, \beta})} \Vert g_2 \Vert_{L^\infty(\R,A_{\alpha, \beta})}, \Vert g_1 \Vert_{L^\infty(\R,A_{\alpha, \beta})} \Vert g_2 \Vert_{M^1(\R,A_{\alpha, \beta})}) \; \Vert \varsigma \Vert_{M^1(\R^2,A_{\alpha, \beta} \otimes \sigma_{\alpha, \beta})}.
\end{eqnarray*}
\end{proposition}
\begin{proof}
Let $\mathcal{K}$ be the function defined on $\R^2$ by 
\begin{equation}\label{eq53}
\mathcal{K}(y, z) = \iint_{\R^2} \varsigma(x, \xi)\; \overline{{g_1}_{x, \xi}^{(\alpha, \beta)} (- z)} \; {g_2}_{x, \xi}^{(\alpha, \beta)} (- y)  \; d(A_{\alpha, \beta} \otimes \sigma_{\alpha, \beta}) (x, \xi).
\end{equation}
Then we define 
\[ \mathfrak{L}_{g_1, g_2}(\varsigma)(f)(y)= \int_{\R}  \mathcal{K}(y, z) f(z) \; A_{\alpha, \beta}(z)\; dz. \]
Now, using Fubini's theorem, for any $z \in \R$, we obtain
\begin{eqnarray*}
&& \int_{\R} | \mathcal{K}(y, z) | \; A_{\alpha, \beta}(y) \;dy \\
&& \leq \int_{\R} \left( \iint_{\R^2} |\varsigma(x, \xi)| \left|\overline{{g_1}_{x, \xi}^{(\alpha, \beta)} (- z)} \right| \left|{g_2}_{x, \xi}^{(\alpha, \beta)} (- y) \right| d(A_{\alpha, \beta} \otimes \sigma_{\alpha, \beta}) (x, \xi) \right) A_{\alpha, \beta}(y) \;dy \\
&& \leq \Vert g_1 \Vert_{L^\infty(\R,A_{\alpha, \beta})} \; \Vert g_2 \Vert_{M^1(\R,A_{\alpha, \beta})} \; \Vert \varsigma \Vert_{M^1(\R^2,A_{\alpha, \beta} \otimes \sigma_{\alpha, \beta})}, 
\end{eqnarray*}
and for any $y \in \R$, we obtain
\[ \int_{\R} | \mathcal{K}(y, z) | \; A_{\alpha, \beta}(z) \;dz  \leq \Vert g_1 \Vert_{M^1(\R,A_{\alpha, \beta})} \; \Vert g_2 \Vert_{L^\infty(\R,A_{\alpha, \beta})} \; \Vert \varsigma \Vert_{M^1(\R^2,A_{\alpha, \beta} \otimes \sigma_{\alpha, \beta})}.   \]
Thus using Schur's lemma (see \cite{fol95}), we conclude that, for $1 \leq p \leq \infty$, $\mathfrak{L}_{g_1, g_2}(\varsigma): M^p(\R,A_{\alpha, \beta}) \to  M^p(\R,A_{\alpha, \beta}) $ is a bounded linear operator, and we have 
\begin{eqnarray*}
&& \Vert \mathfrak{L}_{g_1, g_2}(\varsigma) \Vert_{\mathcal{B}(M^p(\R,A_{\alpha, \beta}))} \\
&& \leq \max(\Vert g_1 \Vert_{M^1(\R,A_{\alpha, \beta})} \Vert g_2 \Vert_{L^\infty(\R,A_{\alpha, \beta})}, \Vert g_1 \Vert_{L^\infty(\R,A_{\alpha, \beta})} \Vert g_2 \Vert_{M^1(\R,A_{\alpha, \beta})}) \; \Vert \varsigma \Vert_{M^1(\R^2,A_{\alpha, \beta} \otimes \sigma_{\alpha, \beta})}.
\end{eqnarray*}
\end{proof} 
 
\begin{remark}
From the Proposition \ref{pro5}, we conclude that the bounded linear operator on $M^p(\R,A_{\alpha, \beta})$, $1 \leq p \leq \infty$, obtained in Proposition \ref{pro4} is actually the integral operator on $M^p(\R,A_{\alpha, \beta})$ with the kernel $\mathcal{K}$ given by (\ref{eq53}).
\end{remark}

\begin{theorem}\label{th3}
Let $\varsigma \in M^p(\R^2,A_{\alpha, \beta} \otimes \sigma_{\alpha, \beta})$, $1 \leq p \leq \infty$ and $g_1, g_2 \in  M^1(\R,A_{\alpha, \beta})$. Then the localization operator  $ \mathfrak{L}_{g_1, g_2}(\varsigma)$ is in $\mathcal{B}(M^q(\R,A_{\alpha, \beta} ))$, $1 \leq q \leq \infty$ and we have 
\[ \Vert \mathfrak{L}_{g_1, g_2}(\varsigma) \Vert_{\mathcal{B}(M^q(\R,A_{\alpha, \beta} ))} \leq \Vert \varsigma \Vert_{M^p(\R^2,A_{\alpha, \beta} \otimes \sigma_{\alpha, \beta})} \; \Vert g_1 \Vert_{M^1(\R,A_{\alpha, \beta})} \; \Vert g_2 \Vert_{M^1(\R,A_{\alpha, \beta})}. \]
\end{theorem}
\begin{proof}
Let $f \in M^q(\R,A_{\alpha, \beta} )$ and $h \in M^{q'}(\R,A_{\alpha, \beta} )$, where $q'$ is the conjugate exponent of $q$. Using the duality between the modulation spaces $ M^p(\R^2,A_{\alpha, \beta} \otimes \sigma_{\alpha, \beta})$ and $M^{p'}(\R^2,A_{\alpha, \beta} \otimes \sigma_{\alpha, \beta})$, we obtain 
\begin{eqnarray}\label{eq47}
\left| \left\langle \mathfrak{L}_{g_1, g_2}(\varsigma)(f), h \right\rangle \right|  
& \leq & \iint_{\R^2} |\varsigma(x, \xi)| \; \left|\W^{(\alpha, \beta)}_{g_1}(f)(x, \xi) \; \overline{\W^{(\alpha, \beta)}_{g_2}(h)(x, \xi)} \right| \; d(A_{\alpha, \beta} \otimes \sigma_{\alpha, \beta}) (x, \xi) \nonumber \\ 
& \leq & \left\Vert \varsigma \right\Vert_{M^p(\R^2,A_{\alpha, \beta} \otimes \sigma_{\alpha, \beta})} \left\Vert \W^{(\alpha, \beta)}_{g_1}(f)  \cdot \overline{\W^{(\alpha, \beta)}_{g_2}(h)} \right\Vert_{M^{p'}(\R^2,A_{\alpha, \beta} \otimes \sigma_{\alpha, \beta})}.
\end{eqnarray}
Since the definition of $M^{p'}(\R^2,A_{\alpha, \beta} \otimes \sigma_{\alpha, \beta})$ is independent of the choice of the window $g$, we estimate the STFT of $\W^{(\alpha, \beta)}_{g_1}(f) \cdot \overline{\W^{(\alpha, \beta)}_{g_2}(h)}$ with respect to some $g \in \mathcal{S}(\R^2) \setminus \{0\}$ with $\Vert g \Vert_{L^{p'}(\R^2,A_{\alpha, \beta} \otimes \sigma_{\alpha, \beta})} \leq 1$. Using Young's convolution inequality and H\"older's inequality, we get
\begin{eqnarray}\label{eq48}
&& \left\Vert \W^{(\alpha, \beta)}_{g_1}(f)  \cdot \overline{\W^{(\alpha, \beta)}_{g_2}(h)} \right\Vert_{M^{p'}(\R^2,A_{\alpha, \beta} \otimes \sigma_{\alpha, \beta})} \nonumber \\
&& = \left\Vert V_g \left( \W^{(\alpha, \beta)}_{g_1}(f) \cdot \overline{\W^{(\alpha, \beta)}_{g_2}(h)} \right) \right\Vert_{L^{p'}(\R^4,A_{\alpha, \beta} \otimes \sigma_{\alpha, \beta})} \nonumber \\
&& =  \left\Vert \left( \W^{(\alpha, \beta)}_{g_1}(f) \cdot \overline{\W^{(\alpha, \beta)}_{g_2}(h)} \right) * M_{\xi} g^*  \right\Vert_{L^{p'}(\R^4,A_{\alpha, \beta} \otimes \sigma_{\alpha, \beta})} \nonumber \\
&& \leq \left\Vert \W^{(\alpha, \beta)}_{g_1}(f) \cdot \overline{\W^{(\alpha, \beta)}_{g_2}(h)} \right\Vert_{L^1(\R^2,A_{\alpha, \beta} \otimes \sigma_{\alpha, \beta})} \left\Vert M_{\xi} g^*  \right\Vert_{L^{p'}(\R^2,A_{\alpha, \beta} \otimes \sigma_{\alpha, \beta})} \nonumber \\ 
&& \leq \left\Vert \W^{(\alpha, \beta)}_{g_1}(f) \right\Vert_{M^q(\R^2,A_{\alpha, \beta} \otimes \sigma_{\alpha, \beta})} \left\Vert \W^{(\alpha, \beta)}_{g_2}(h) \right\Vert_{M^{q'}(\R^2,A_{\alpha, \beta} \otimes \sigma_{\alpha, \beta})} \nonumber \\
&& = \left\Vert f *_{\alpha, \beta} \overline{\mathcal{M}^{(\alpha, \beta)}_\xi g_1} \right\Vert_{M^q(\R^2,A_{\alpha, \beta} \otimes \sigma_{\alpha, \beta})} \left\Vert h *_{\alpha, \beta} \overline{\mathcal{M}^{(\alpha, \beta)}_\xi g_2} \right\Vert_{M^{q'}(\R^2,A_{\alpha, \beta} \otimes \sigma_{\alpha, \beta})} \nonumber \\
&& \leq \Vert f \Vert_{M^q(\R,A_{\alpha, \beta})} \; \Vert g_1 \Vert_{M^1(\R,A_{\alpha, \beta})} \; \Vert h  \Vert_{M^{q'}(\R,A_{\alpha, \beta})} \; \Vert g_2 \Vert_{M^1(\R,A_{\alpha, \beta})}.
\end{eqnarray}
Thus from (\ref{eq47}) and (\ref{eq48}), we obtain
\[ \Vert \mathfrak{L}_{g_1, g_2}(\varsigma) \Vert_{\mathcal{B}(M^q(\R,A_{\alpha, \beta} ))} \leq \Vert \varsigma \Vert_{M^p(\R^2,A_{\alpha, \beta} \otimes \sigma_{\alpha, \beta})} \; \Vert g_1 \Vert_{M^1(\R,A_{\alpha, \beta})} \; \Vert g_2 \Vert_{M^1(\R,A_{\alpha, \beta})}. \]
\end{proof}

\subsection{Compactness of $\mathfrak{L}_{g_1, g_2}(\varsigma)$ for symbols in $M^1(\R^2,A_{\alpha, \beta} \otimes \sigma_{\alpha, \beta})$}

In this section, we prove that the localization operators $\mathfrak{L}_{g_1, g_2}(\varsigma): M^p(\R,A_{\alpha, \beta}) \to M^p(\R,A_{\alpha, \beta}), \; 1<p< \infty$ are compact for symbols $\varsigma$ in $M^1(\R^2,A_{\alpha, \beta} \otimes \sigma_{\alpha, \beta})$. Let us start with the following proposition.

\begin{proposition}
Let $\varsigma \in M^1(\R^2,A_{\alpha, \beta} \otimes \sigma_{\alpha, \beta})$ and $g_1, g_2 \in  M^1(\R,A_{\alpha, \beta}) \cap L^\infty(\R,A_{\alpha, \beta})$. Then, for fixed $1<p< \infty$, the localization operator $\mathfrak{L}_{g_1, g_2}(\varsigma): M^p(\R,A_{\alpha, \beta}) \to M^p(\R,A_{\alpha, \beta})$ is compact.
\end{proposition}
\begin{proof}
Let $\{ f_n \}_{n \in \mathbb{N}}$ be a sequence of functions in $M^p(\R,A_{\alpha, \beta})$ such that $f_n \rightharpoonup 0$ weakly in $M^p(\R,A_{\alpha, \beta})$ as $n \to \infty$. It is sufficient to prove that $\lim\limits_{n \to \infty} \Vert \mathfrak{L}_{g_1, g_2}(\varsigma)(f_n) \Vert_{M^p(\R,A_{\alpha, \beta})}=0$. From the relation (\ref{eq41}), we obtain 
\begin{eqnarray}\label{eq54}
&& \left| \mathfrak{L}_{g_1, g_2}(\varsigma)(f_n)(y) \right| \nonumber \\
&& \leq \iint_{\R^2} |\varsigma(x, \xi)| \left| \left\langle f_n, {g_1}_{x, \xi}^{(\alpha, \beta)} (- \; \cdot) \right\rangle_{L^2(\R,A_{\alpha, \beta})} \right|  \left| {g_2}_{x, \xi}^{(\alpha, \beta)} (- y) \right| \; d(A_{\alpha, \beta} \otimes \sigma_{\alpha, \beta})(x,\xi).
\end{eqnarray}
Since $f_n \rightharpoonup 0$ weakly in $M^p(\R,A_{\alpha, \beta})$, we deduce that
\begin{equation}\label{eq55}
\lim_{n \to \infty} |\varsigma(x, \xi)| \left| \left\langle f_n, {g_1}_{x, \xi}^{(\alpha, \beta)} (- \; \cdot) \right\rangle_{L^2(\R,A_{\alpha, \beta})} \right|  \left| {g_2}_{x, \xi}^{(\alpha, \beta)} (- y) \right|=0, \quad  \text{for all} \;\; x, y, \xi \in \R.
\end{equation}
Moreover, as $f_n \rightharpoonup 0$ weakly in $M^p(\R,A_{\alpha, \beta})$ as $n \to \infty$, then there exists a constant $C>0$ such that $\Vert f_n \Vert_{M^p(\R,A_{\alpha, \beta})} \leq C$. Hence, for all  $x, y, \xi \in \R$, we obtain 
\begin{equation}\label{eq56}
|\varsigma(x, \xi)| \left| \left\langle f_n, {g_1}_{x, \xi}^{(\alpha, \beta)} (- \; \cdot) \right\rangle_{L^2(\R,A_{\alpha, \beta})} \right|  \left| {g_2}_{x, \xi}^{(\alpha, \beta)} (- y) \right|  \leq C |\varsigma(x, \xi)| \; \Vert g_1 \Vert_{M^\infty(\R,A_{\alpha, \beta})} \left| {g_2}_{x, \xi}^{(\alpha, \beta)} (- y) \right|.
\end{equation}
Further, using Fubini's theorem, we get
\begin{eqnarray}\label{eq57}
&& \Vert \mathfrak{L}_{g_1, g_2}(\varsigma)(f_n) \Vert_{M^p(\R,A_{\alpha, \beta})} \nonumber \\
&& = \left\| V_g \left( \mathfrak{L}_{g_{1}, g_{2}}(\varsigma)\left(f_{n}\right) \right) \right\|_{L^{p}\left(\mathbb{R}^2, A_{\alpha, \beta}\right)} \nonumber \\
&& =\left\|  \left( \mathfrak{L}_{g_{1}, g_{2}}(\varsigma)\left(f_{n}\right)   \right) * M_{\xi} g^{*} \right\|_{L^{p}\left(\mathbb{R}^2, A_{\alpha, \beta}\right)} \nonumber \\
&& \leq \left\|  \mathfrak{L}_{g_{1}, g_{2}}(\varsigma)\left(f_{n}\right) \right\|_{L^{1}\left(\mathbb{R}, A_{\alpha, \beta}\right)} \left\| M_{\xi} g^{*} \right\|_{L^{p}\left(\mathbb{R}, A_{\alpha, \beta}\right)} \nonumber \\
&& \leq  \int_{\R} \iint_{\R^2} |\varsigma(x, \xi)| \left| \left\langle f_n, {g_1}_{x, \xi}^{(\alpha, \beta)} (- \; \cdot) \right\rangle_{L^2(\R,A_{\alpha, \beta})} \right|  \left| {g_2}_{x, \xi}^{(\alpha, \beta)} (- y) \right| \; d(A_{\alpha, \beta} \otimes \sigma_{\alpha, \beta})(x,\xi) \; A_{\alpha, \beta}(y) \;dy  \nonumber \\
&& \leq C \; \Vert g_1 \Vert_{M^\infty(\R,A_{\alpha, \beta})} \iint_{\R^2} |\varsigma(x, \xi)|  \int_{\R} \left| {g_2}_{x, \xi}^{(\alpha, \beta)} (- y) \right| \; A_{\alpha, \beta}(y) dy \; d(A_{\alpha, \beta} \otimes \sigma_{\alpha, \beta})(x,\xi) \nonumber \\
&& \leq C \; \Vert g_1 \Vert_{L^\infty(\R,A_{\alpha, \beta})} \; \Vert g_2 \Vert_{M^1(\R,A_{\alpha, \beta})} \; \Vert \varsigma \Vert_{M^1(\R^2,A_{\alpha, \beta} \otimes \sigma_{\alpha, \beta})} < \infty. 
\end{eqnarray}
Thus, using the Lebesgue dominated convergence theorem and the relations (\ref{eq54}) -- (\ref{eq57}), we obtain that 
\[\lim\limits_{n \to \infty} \Vert \mathfrak{L}_{g_1, g_2}(\varsigma)(f_n) \Vert_{M^p(\R,A_{\alpha, \beta})}=0.  \]
This completes the proof. 
\end{proof}

\begin{theorem}
Let $\varsigma \in M^1(\R^2,A_{\alpha, \beta} \otimes \sigma_{\alpha, \beta})$,  $g_1 \in  M^{p'}(\R,A_{\alpha, \beta})$ and $g_2 \in  M^p(\R,A_{\alpha, \beta})$, for $1 < p < \infty$. Then the localization operator $\mathfrak{L}_{g_1, g_2}(\varsigma): M^p(\R,A_{\alpha, \beta}) \to M^p(\R,A_{\alpha, \beta})$ is compact.
\end{theorem}
\begin{proof}
Let $p'$ be the conjugate exponent of $p$. We first show that the conclusion of the previous proposition holds for $p'$. The operator $\mathfrak{L}_{g_1, g_2}(\varsigma): M^{p'}(\R,A_{\alpha, \beta}) \to M^{p'}(\R,A_{\alpha, \beta})$ is the adjoint of the operator $\mathfrak{L}_{g_1, g_2}(\overline{\varsigma}): M^p(\R,A_{\alpha, \beta}) \to M^p(\R,A_{\alpha, \beta})$, which is compact by the previous proposition. Hence, by the duality properties of modulation spaces, $\mathfrak{L}_{g_1, g_2}(\varsigma): M^{p'}(\R,A_{\alpha, \beta}) \to M^{p'}(\R,A_{\alpha, \beta})$ is compact. Finally, using the interpolation of the compactness on $M^p(\R,A_{\alpha, \beta})$ and on $M^{p'}(\R,A_{\alpha, \beta})$, the proof is complete. 
\end{proof}

\section*{Acknowledgments}
The author is deeply indebted to Prof. Nir Lev for several fruitful discussions and generous comments. The author wishes to thank the anonymous referees for their helpful
comments and suggestions that helped to improve the quality of the paper.

\end{document}